\documentclass[11pt]{amsart}

\usepackage{etex}
\usepackage{graphicx,amsthm}
\usepackage{enumerate}
\usepackage{empheq}
\usepackage{amssymb}
\usepackage{pst-all}
\usepackage{array}
\usepackage{tikz}
\usepackage{tikz-cd}
\usepackage{amsmath,amscd,amsfonts,epic,eepic,bbm, mathabx,graphicx,ytableau}
\allowdisplaybreaks[4]

\usepackage{a4wide}

\usepackage[all]{xy} 

\usepackage{caption}
\usepackage{subcaption}
\usepackage{url}

\usepackage[mathscr]{euscript} 
\usepackage{booktabs} 
\usepackage{ amssymb }
\usetikzlibrary{decorations.markings, decorations.pathreplacing, calc, tikzmark,positioning} 
\usepackage{hyperref}

\usepackage{colonequals}
\usepackage{ragged2e}

\newtheorem{theorem}{Theorem}[section]
\newtheorem{lemma}[theorem]{Lemma}
\newtheorem{proposition}[theorem]{Proposition}
\newtheorem{corollary}[theorem]{Corollary}

\theoremstyle{definition}
\newtheorem{definition}[theorem]{Definition}
\newtheorem{example}[theorem]{Example}
\newtheorem{question}[theorem]{Question}

\theoremstyle{remark}
\newtheorem{remark}[theorem]{Remark}
\newtheorem{conjecture}[theorem]{Conjecture}

\numberwithin{equation}{section}

\newtheorem{thmy}{Theorem}

\newcommand{\C}{\mathbb{C}}

\newcommand{\Cstar}{\mathbb{C}^{\ast}}
\newcommand{\flag}{F\ell}
\newcommand{\Xwo}[1]{X_{#1}^{\circ}}

\newcommand{\Owo}[1]{\Omega_{#1}^{\circ}}
\newcommand{\Ow}[1]{\Omega_{#1}}
\newcommand{\Owho}[1]{\Omega_{#1,h}^{\circ}}
\newcommand{\Owh}[1]{\Omega_{#1,h}}

\newcommand{\swh}[1]{\sigma_{#1,h}}
\DeclareMathOperator{\Hess}{Hess}

\DeclareMathOperator{\codim}{\codim}

\DeclareMathOperator{\des}{des}
\DeclareMathOperator{\Des}{Des}
\DeclareMathOperator{\Stab}{Stab}

\makeatletter
\let\@@span\span
\def\sp@n{\@@span\omit\advance\@multicnt\m@ne}
\makeatother
\renewcommand{\span}{\text{span}}

\newcommand{\wi}[1]{w^{[#1]}}

\newcommand{\Gwh}[1]{G_{#1,h}}
\allowdisplaybreaks

\newcommand{\Ghone}{\mathcal{G}_h^1}
\newcommand{\hats}{\widehat{\sigma}}
\newcommand{\hatswh}[1]{\hats_{#1,h}}

\newcommand{\vred}{\textcolor{red}{|}}

\newcommand{\s}{\mathsf{s}}

\newcommand{\la}{\lambda}

\newlength{\celldim} 
\newsavebox{\cell}

\sbox{\cell}{%
\begin{picture}(22,22)\linethickness{0.6pt} %
  \put(0,0){\line(1,0){22}} \put(0,0){\line(0,1){22}}
  \put(22,0){\line(0,1){22}} \put(0,22){\line(1,0){22}}
\end{picture}}

\newcommand\cellifying[1]{%
  \def\thearg{#1}\def\nothing{}%
  \ifx\thearg\nothing \vrule width0pt height\celldim depth0pt\else
  \hbox to 0pt{\usebox{\cell} \hss}\fi%
  \vbox to \celldim{ \vss \hbox to
  \celldim{\hss$#1$\hss} \vss}
}
\newcommand\ttableau[1]{\vtop{\let\\\cr
\baselineskip -16000pt \lineskiplimit 16000pt \lineskip 0pt
\ialign{&\cellifying{##}\cr#1\crcr}}}

\begin{document}
	
\title[Permutation module decomposition of $H^2(\Hess(S,h))$]{Permutation module decomposition of the second cohomology of  a regular semisimple Hessenberg variety}

\author{Soojin Cho}
\address{Department of Mathematics, Ajou University, Suwon  16499, Republic of Korea}
\email{chosj@ajou.ac.kr}
	
\author{Jaehyun Hong}
\address{Center for Complex Geometry, Institute for Basic Science (IBS), Daejeon 34126, Republic of Korea}
\email{jhhong00@ibs.re.kr}
	
\author{Eunjeong Lee}
\address{Department of Mathematics, Chungbuk National University, Cheongju 28644, Republic of Korea}
\email{eunjeong.lee@chungbuk.ac.kr}
	
\thanks{Cho was supported by NRF-2020R1A2C1A01011045.	Hong was supported by the Institute for Basic Science (IBS-R032-D1). Lee was supported by the Institute for Basic Science (IBS-R003-D1).}

\begin{abstract}
Regular semisimple Hessenberg varieties admit actions of associated Weyl groups on their cohomology spaces of each degree. In this paper, we consider the module structure of the cohomology spaces of regular semisimple Hessenberg varieties of type $A$. We define a subset of the Bia{\l}ynicki-Birula basis of the cohomology space which becomes a module generator set of the cohomology module of each degree. We use these generators to construct permutation submodules of the degree two cohomology module to form a permutation module decomposition. Our construction is consistent with a known combinatorial result by Chow on chromatic quasisymmetric functions.
\end{abstract}


\maketitle

\setcounter{tocdepth}{1}
\tableofcontents
\section{Introduction} \label{sec:intro}
Since De Mary,   Proceci, and  Shayman  introduced
Hessenberg varieties in the 1990s (\cite{DeMari_Shayman_88} and \cite{DPS_Hessenberg_var}), many researchers in various fields have increasingly focused on them.
Hessenberg varieties form a family of subvarieties of the full flag varieties and many interesting varieties appear as Hessenberg varieties. For example, full flag varieties and permutohedral varieties are Hessenberg varieties.
Flag varieties are central objects in the intersection of algebraic geometry and algebraic combinatorics.  Hessenberg varieties share similar features with flag varieties.

A \emph{regular semisimple Hessenberg variety} $\Hess(S,h)$ is a subvariety of the full flag variety~$\flag(\C^n)$ which is determined by two data: a weakly increasing function $h \colon [n] \to [n]$, called a \emph{Hessenberg function},  and a regular semisimple linear operator $S$ (see Definition~\ref{def_Hess_var} for a precise definition). Here, we use $[n]$ to denote the set $\{1, \dots, n\}$. Tymoczko defined an action of the symmetric group~$\mathfrak{S}_n$ on the cohomology space $H^{2k}(\Hess(S,h);\C)$ in~\cite{T2, T1}, which is called the \emph{dot action}.\footnote{In what follows, we consider cohomology rings with coefficients in $\C$ and we will not indicate the coefficient ring that we are working on.}
On the other hand, a Hessenberg function $h$ determines a graph~$G_h$, called
the \emph{incomparability graph}, of the corresponding unit interval order.

Shareshian and Wachs~\cite{SW} refined a long standing conjecture proposed by Stanley and Stembridge~\cite{SS, S1} on chromatic symmetric functions as a conjecture on chromatic quasisymmetric functions.

\begin{conjecture}[\cite{SS, S1}, \cite{SW}]\label{conj:e-positivity}
Let $h \colon [n] \to [n]$ be a Hessenberg function and let $\Hess(S, h)$ be the regular semisimple Hessenberg variety associated with $h$ and a regular semisimple linear operator $S$. Let $X_{G_h}(\mathbf{x}, t)$ be the chromatic quasisymmetric function of the graph $G_h$ associated with $h$.
Then, for $0 \leq k \leq N_h \coloneq \sum_{i=1}^n (h(i)-i)$,
  the coefficient of $t^k$ of 
   $X_{G_h}({\bf x},t)$ is positively expanded as a sum of elementary symmetric functions.
   \end{conjecture}

Furthermore, they proposed another conjecture in the same paper~\cite{SW} that Conjecture~\ref{conj:e-positivity} is equivalent to a conjecture regarding the $\mathfrak S_n$-module structure of $H^*(\Hess(S,h))$. 
The latter conjecture has been proved by Brosnan and Chow~\cite{BrosnanChow18}, and by Guay-Paquet~\cite{G-P2}.

\begin{theorem}[Conjectured in \cite{SW}; proved in \cite{BrosnanChow18}, \cite{G-P2}]\label{thm:Brosnan-Chow}
Let $h \colon [n] \to [n]$ be a Hessenberg function and let $\Hess(S, h)$ be the regular semisimple Hessenberg variety associated with $h$ and a regular semisimple linear operator $S$. Let $X_{G_h}(\mathbf{x}, t)$ be the chromatic quasisymmetric function of the graph $G_h$ associated with $h$. Then, we have
\[
\sum_{k=0}^{N_h} \mathrm{ch} H^{2k}(\Hess(S, h))\,t^k=\omega X_{G_h}(\mathbf x, t)\,,\]
where $\mathrm{ch}$ is the Frobenius characteristic map and $\omega$ is the involution on the symmetric function algebra sending elementary symmetric functions to complete homogeneous symmetric functions.
\end{theorem}
%


To a partition $\la=(\la_1, \dots, \la_{\ell})$ of $n$ we associate a module $M^\la$  of $\mathfrak{S}_n$, called a {\it permutation module of type} $\lambda$, which is defined by  the vector space of formal linear sums of the ordered tuples  $(J_1, \dots, J_{\ell})$ of disjoint subsets of $\{1, 2, \dots, n\}$ satisfying $\lvert J_s \rvert=\la_s$, $s=1, \dots, \ell$, and $\lvert \bigcup_s J_s \rvert =n$, where the permutations in $\mathfrak{S}_n$ act on $(J_1, \dots, J_{\ell})$ naturally.
Under the Frobenius characteristic map, $M^\la$ corresponds to the complete homogeneous symmetric function $h_\la$.  
In Theorem~\ref{thm:Brosnan-Chow} and throughout the paper, by a  {\it permutation module} we mean   a permutation module $M^\la$    of type $\lambda$ for some partition $\lambda$ of $n$.\footnote{In the literature, a permutation module or a permutation representation of the symmetric group $\mathfrak{S}_n$ is a representation of $\mathfrak S_n$ with a basis permuted by the action of $\mathfrak S_n$; or more restrictively,  a representation of $\mathfrak{S}_n$, in which every point stabilizer is parabolic, in addition. The condition that we impose in this paper is the most restrictive one.}

%


Theorem~\ref{thm:Brosnan-Chow} inspired efforts to understand combinatorial works on chromatic quasisymmetric functions in the language of Hessenberg varieties, and to prove Conjecture~\ref{conj:e-positivity} using geometric methods. One of them is to construct nice bases $\{\swh{w}\}_{w\in\mathfrak{S}_n}$ of $H^*(\Hess(S,h))$ using Bia{\l}ynicki-Birula cell decompositions of the Hessenberg varieties, called \emph{Bia{\l}ynicki-Birula bases} (\emph{BB bases}) (see Definition \ref{def_swh_classes}), and to investigate their properties; see~\cite{CHL}.

Our work in this paper provides a general method of using the BB bases to construct a permutation module decomposition of  $H^*(\Hess(S,h))$. More specifically, we define a subset~$\mathcal G_h^k$ of $\frak S_n$ for a Hessenberg function  $h\colon [n] \rightarrow [n]$;
\[\mathcal G_h^k\colonequals  \{ w \in \mathfrak{S}_n \mid \ell_h(w)=k,\,\, w^{-1}(w(j)+1) \leq h(j) \,\, \text{ for } w(j) \in [n-1] \}\,.\]
We then prove that the corresponding BB basis elements form a module generator set of the cohomology of degree $2k$. We use related results in both combinatorics of chromatic quasisymmetric functions and geometry of Hessenberg varieties.

\begin{thmy}[Theorem~\ref{thm:module generator}]\label{thmx:A}
The set $\{\sigma_{w,h} \mid w \in \mathcal G_h^k\}$ has the cardinality $\dim_{\C} H^{2k}(\Hess(S,h))^{\mathfrak S_n}$ and it generates the  $\mathfrak S_n$-module  $H^{2k}(\Hess(S,h))$, that is, 
\[
H^{2k}(\Hess(S,h)) = \sum_{w \in \mathcal G_h^k} M(\sigma_{w,h}),
\] 
where $M(\sigma_{w,h})$ denotes the $\mathfrak S_n$-module generated by $\sigma_{w,h}$.
\end{thmy}

We remark that the sum in Theorem~\ref{thmx:A} is not a direct sum, the question is whether we can reduce modules~$M(\swh{w})$ to get a direct sum decomposition of permutation modules as stated in the following conjecture.

\begin{conjecture}\label{conj:permutation module decomposition_intro}
For each $w \in \mathcal G_h^k$, there is $\widehat{\sigma}_{w,h} \in M(\sigma_{w,h})$ such that the $\mathfrak S_n$-module $M(\widehat{\sigma}_{w,h})$ generated by $\widehat{\sigma}_{w,h}$ is a permutation module and
\[
H^{2k}(\Hess(S,h)) = \bigoplus_{w \in \mathcal G_h^k} M(\widehat{\sigma}_{w,h}). 
\]
Here, for a precise description of a candidate of $\widehat{\sigma}_{w,h}$, see Definition~\ref{def_widehat_swh}.
\end{conjecture}

In \cite{CHL}, Conjecture~\ref{conj:permutation module decomposition_intro} was proved for permutohedral varieties, which, in turn, provided a geometric proof of Conjecture~\ref{conj:e-positivity}. In this paper, we use generators $\sigma_{w,h}$, $w \in \mathcal G_h^1$, to construct mutually disjoint permutation submodules of $H^{2}(\Hess(S,h))$ that constitute the entire module $H^{2}(\Hess(S,h))$; that is, we prove Conjecture~\ref{conj:permutation module decomposition_intro} when $k=1$.

\begin{thmy}[Theorem~\ref{thm_H2}]
For each $w \in \mathcal G_h^1$, there is $\widehat{\sigma}_{w,h} \in M(\sigma_{w,h})$ such that the $\mathfrak S_n$-module~$M(\widehat{\sigma}_{w,h})$ generated by $\widehat{\sigma}_{w,h}$ is a permutation module and
\[
H^{2}(\Hess(S,h)) = \bigoplus_{w \in \mathcal G_h^1} M(\widehat{\sigma}_{w,h}). 
\]
\end{thmy}
This provides a \emph{geometric} proof of a known result of Conjecture~\ref{conj:e-positivity} on chromatic quasisymmetric functions when $k=1$, done by Chow in \cite{Chow_h2}; see Theorem~\ref{thm:Chow}.

Our paper is structured as follows. In Section~\ref{sec:preliminaries}, we recall the basics on regular  semisimple Hessenberg varieties and symmetric group action on their cohomology spaces; especially focusing  on the results in \cite{CHL}. We define a special set $\{\sigma_{w,h} \mid w \in \mathcal G_h^k\}$
 of module  generators of $H^{2k}(\Hess(S,h))$ in Section~\ref{sec:module generators}. After restricting ourselves to $k = 1$, we
classify the elements in~$\mathcal G_h^1$ and compute the stabilizer subgroup of the generators~$\swh{w}$ for $w\in \mathcal G_h^1$ in Section~\ref{sec:generators of H2}.
In Section~\ref{sec_partition}, we consider a partition of $\mathfrak{S}_n$ defined by a certain equivalence relation such that there is a bijective correspondence between the set $\mathcal G_h^1$ and the equivalence classes. Furthermore, we enumerate the cardinality of each class and the dimension of $H^{2}(\Hess(S,h))$.
Section~\ref{sec_geometric_construction} is devoted to construct a permutation module decomposition of $H^{2}(\Hess(S,h))$.
In the final section, we suggest a more refined conjecture.

\section{Preliminaries}\label{sec:preliminaries}

\subsection{Hessenberg varieties}
\label{sec_preliminaries_1}

For a positive integer $n$, we use
$\mathfrak{S}_n$ to denote the symmetric group on $[n]$.
Hessenberg varieties are certain subvarieties of the \emph{full flag variety} $\flag(\C^n)$;
\[
\flag(\C^n) \colonequals  \{ V_{\bullet} = (\{0\} \subsetneq V_1 \subsetneq V_2 \subsetneq \cdots \subsetneq V_{n-1} \subsetneq \C^n) \mid \dim_{\C} V_i = i \,\,\, \text{ for all }  1\leq i \leq n\}.
\]

Hessenberg varieties in $\flag(\C^n)$ are parametrized by $n \times n$ matrices, and \emph{Hessenberg functions} $h \colon [n] \rightarrow [n]$ satisfying the following two properties
\begin{itemize}
\item $h(1)\leq h(2) \leq \cdots \leq h(n)$, and
\item $i\leq h(i)$ for all $i=1, \dots, n$.
\end{itemize}
A Hessenberg function is frequently described by writing its values in a parenthesis; $h=(h(1), \dots, h(n))$.

\begin{definition}\label{def_Hess_var}
	Let $h$ be a Hessenberg function and let $S$ be a regular semisimple linear operator with $n$ distinct eigenvalues.
	The regular semisimple \textit{Hessenberg variety} $\Hess(S,h)$ is a subvariety of the flag variety defined as
	\[
	\Hess(S,h) = \{ V_{\bullet} \in \flag(\C^n) \mid S V_{i} \subset V_{h(i)} \,\,\, \text{ for all }1 \leq i \leq n \}.
	\]
\end{definition}

We note that if $h(i) \geq i +1$ for all $1 \leq i \leq n-1$, then $\Hess(S,h)$ is irreducible. Because of this, we assume $h(i) \geq i+1$ for all $1 \leq i \leq n-1$ throughout this paper.
The notion of the length $\ell(w)$ of a permutation $w\in \mathfrak{S}_n$, which counts the number of inversions of $w$, is extended to a quantity respecting a given Hessenberg function $h$;
\[\ell_h(w) \colonequals  \lvert\{ (j,i) \mid 1 \leq j < i \leq h(j), w(j) > w(i)\}\rvert. \]
Note that $\ell(w)=\ell_h(w)$ for $h=(n, n, \dots, n)$.

Some fundamental properties of the Hessenberg varieties are stated in the following proposition.
See Section 2 of \cite{CHL} and the references therein for more details.

\begin{proposition}\label{prop_basics}
Let $S$ be a regular semisimple linear operator and let $h \colon [n] \to [n]$ be a Hessenberg function satisfying $h(i) \geq i+1$ \textup{(}$1 \leq i \leq n-1$\textup{)}.
\begin{enumerate}

\item $\Hess(S,h)$ is a smooth variety of $\C$-dimension $\sum_{i=1}^n (h(i) - i)$.
\item The complex torus $T \colonequals  (\Cstar)^n$ acts on $\Hess(S,h)$ by left multiplication, and the fixed points of this action can be identified with the permutations in $\mathfrak{S}_n$.
\item\label{BB_decomp} \textup{(}Bia{\l}ynicki-Birula decomposition\textup{)} $\Hess(S,h)$ is decomposed into their plus cells $X_{w,h}^{\circ}$ and also into their minus cells $\Owho{w}$;
\[\Hess(S,h) =\bigsqcup_{w \in \mathfrak{S}_n} X_{w,h}^{\circ}= \bigsqcup_{w \in \mathfrak{S}_n} \Owho{w}\,.\]
The plus cell $X_{w,h}^{\circ}$ \textup{(}the minus cell $\Omega_{w,h}^{\circ}$, respectively\textup{)} is the intersection of $\Hess(S,h)$ with the Schubert cell $X_w^{\circ}$ \textup{(}the opposite Schubert cell $\Omega_w^{\circ}$, respectively\textup{)} of $\flag(\C^n)$.

\item\label{dim} The dimension of plus cells and minus cells are given by
\[ \dim_{\C} X_{w,h}^{\circ} = \ell_h(w) \,\,\,\mbox{ and }\,\,\, \dim_{\C} \Owho{w}= \dim_{\C}(\Hess(S,h)) -\ell_h(w)\,, \] and therefore the Poincar\'e polynomial $\mathscr{P}\textup{oin}(\Hess(S,h), q)$ is given by
	\[
	\mathscr{P}\textup{oin}(\Hess(S,h) ,q )
	\colonequals  \sum_{k \geq 0} \dim_{\mathbb{C}} H^{k}(\Hess(S,h))  q^k
	= \sum_{w \in \mathfrak{S}_n} q^{2\ell_h(w)}.
	\]
\item For any regular semisimple linear operator $S'$, $H^*(\Hess(S,h))\cong H^*(\Hess(S',h))$ as graded $\C$-algebras.
\end{enumerate}
\end{proposition}

A regular semisimple Hessenberg variety $\Hess(S,h)$ is a GKM manifold \cite{T2, T1}, and the equivariant cohomology ring of $\Hess(S,h)$ can be described in terms of its GKM graph~$(V, E, \alpha)$ as a subring of the direct sum of polynomial rings $\bigoplus_{v \in V} \C[t_1,\dots,t_n]$. We use $s_{j, i}$ to denote the transposition in $\mathfrak{S}_n$ that exchanges $j$ and $i$ for $1 \leq j < i \leq n$.

\begin{theorem}[{\cite{GKM98}}]\label{thm_GKM}
Let $\Hess(S,h)$ be a regular semisimple Hessenberg variety.
The GKM graph $(V, E, \alpha)$ of $\Hess(S,h)$ comprises the set $E = \mathfrak{S}_n$ of vertices, the set of \textup{(}directed\textup{)} edges $E=\{ (v\rightarrow w)\mid w=vs_{j, i}, \mbox{ for } j<i\leq h(j)\}$, and the edge labeling $\alpha$ such that $\alpha(v\rightarrow vs_{j, i})=t_{v(i)}-t_{v(j)}\in \C[t_1,\dots,t_n]$. Moreover,
	\[
	H^{\ast}_T(\Hess(S,h)) \cong
	\left\{
	(p(v)) \in \bigoplus_{v \in \mathfrak{S}_n} \C[t_1,\dots,t_n] \,\middle\vert\, \alpha{(v \to w)} \mid (p(v) - p(w)) \, \text{ for all }(v \to w) \in E
	\right\}.
	\]
\end{theorem}
We note that for the set $E$ of edges in the statement, we have $(v \to w) \in E$ if and only if $(w \to v) \in E$. In particular, $\alpha(v \to w) = -\alpha(w \to v)$.

The main concern of this article is on the structure of the cohomology $H^{\ast}(\Hess(S,h))$ which can be described from Theorem~\ref{thm_GKM} using the following ring isomorphism;
\[
H^{\ast}(\Hess(S,h))\cong  H^{\ast}_T(\Hess(S,h))/(t_1, \dots, t_n) \,,
\]
where we use $t_i$ to indicate the element in $H^{\ast}_T(\Hess(S,h))$ whose value at each $v \in \mathfrak{S}_n$ is $t_i$.

\subsection{\texorpdfstring{Bia{\l}ynicki-Birula}{Bialynicki-Birula} basis and symmetric group action}
\label{sec_preliminaries_2}

The closure $\Owh{w}  \colonequals  \overline{\Owho{w}}$ of a minus cell $\Owho{w}$ defines a class $[\Owh{w}]$ in the equivariant Chow ring $A^{\ast}_T(\Hess(S,h))$ graded by the codimension,  and the cycle map from the Chow ring to the cohomology ring is an isomorphism
\[
cl_{\Hess(S,h)}^T \colon A_{T}^{\ast}(\Hess(S,h))\stackrel{\cong}{\longrightarrow} H_T^{2\ast}(\Hess(S,h)).
\]

\begin{definition}[{\cite[Definition~2.9]{CHL}}]\label{def_swh_classes}
	Let $\Hess(S,h)$ be a regular semisimple Hessenberg variety.
	For $w \in \mathfrak{S}_n$, the \emph{Bia{\l}ynicki-Birula class \textup{(}BB class\textup{)}} $\swh{w}^T \in H^{2 \ell_h(w)}_{T}(\Hess(S,h))$ is the image of the class 
\[
[\Owh{w}] \in A^{\ell_h(w)}_T(\Hess(S,h))
\] 
under the cycle map $cl_{\Hess(S,h)}^T$.
 We let $\sigma_{w, h}\in H^{2 \ell_h(w)}(\Hess(S,h))$ be the corresponding class of $\swh{w}^T \in H^{2 \ell_h(w)}_{T}(\Hess(S,h))$.
\end{definition}

Because of Proposition~\ref{prop_basics}\eqref{BB_decomp}, the BB classes $\swh{w}^T$ ($\swh{w}$, respectively) form a basis of the equivariant cohomology space (ordinary cohomology space, respectively) of $\Hess(S,h)$.

\begin{proposition}\label{prop_basis}
Let $\Hess(S,h)$ be a regular semisimple Hessenberg variety. Then the classes~$ \swh{w}^T$, $w  \in \mathfrak{S}_n $, form a basis called the \emph{Bia{\l}ynicki-Birula basis (BB basis)} of $H^{\ast}_T(\Hess(S,h))$, and the classes $ \swh{w}$, $w  \in \mathfrak{S}_n $, form a basis called the \emph{Bia{\l}ynicki-Birula basis (BB basis)} of~$H^{\ast}(\Hess(S,h))$.
\end{proposition}

In \cite{T2, T1}, Tymoczko defined an action called the \emph{dot action} of the symmetric group $\mathfrak S_n$   on the equivariant cohomology $H_T^*(\Hess(S,h))$ as follows.
For $u \in \mathfrak S_n$ and $\sigma=(\sigma(v))_{v \in \mathfrak S_n} \in H_T^*(\Hess(S,h)) \subset \bigoplus_{v \in \mathfrak S_n} \mathbb C[t_1, \dots, t_n]$,
\[
(u\cdot \sigma)(v) \colonequals
(\sigma(u^{-1}v))(t_{u(1)}, \dots, t_{u(n)}).
\]
Since the ideal $(t_1,\dots,t_n) \subset H^{\ast}_T(\Hess(S,h))$ is  invariant under the dot action on $H^{\ast}_T(\Hess(S,h))$, it induces an action on $H^{\ast}(\Hess(S,h))$ which we call the \emph{dot action} as well.


We review the description of $s_i \cdot \swh{w}$ provided in~\cite[Section~4]{CHL}.
For two permutations~$v, w \in \mathfrak S_n$ such that $w=v s_{i, j}$ and $\ell(v) >\ell(w)$, we use $v \rightarrow w$ to mean that $(v \rightarrow w)$ is an edge of the GKM graph of $\Hess(S,h)$; and $v \dasharrow w$ to mean that  $(v \rightarrow w)$ is not an edge of the GKM graph of $\Hess(S,h)$.

Let $w$ be a permutation in $\mathfrak S_n$ and let $s_i=s_{i, i+1}$ be a simple reflection such that $w \rightarrow s_i w$. The Bia{\l}ynicki-Birula decomposition of $\Omega_{s_{i}w,h}$ is given by
\[
\Omega_{s_{i}w,h}= \bigsqcup_{u \in \Omega_{s_{i}w,h}^T}(\Omega_u^{\circ} \cap \Omega_{s_{i}w,h}). 
\]
Here, for a given $T$-invariant variety~$X$, we denote by $X^T$ the set of $T$-fixed points.

Define $\mathcal A_{s_i,w}$ by the set of all $u \in    \Omega_{s_{i}w,h}^T \cap \Omega_w^T$ such that  $\dim_{\C} (\Omega_u^{\circ} \cap \Omega_{s_{i}w,h})  =\dim_{\C} \Omega_{w,h}$ and  $u \dashrightarrow s_iu$.
For $u \in \mathcal A_{s_i,w}$, define $\mathcal T_u$ and $\mathcal T_{s_iu}$  by the closures of $   \Omega_u^{\circ} \cap \Omega_{s_{i}w,h}$ and $   \Omega_{s_iu}^{\circ} \cap \Omega_{s_{i}w,h}$, and
let $\tau_u$ and $\tau_{s_iu}$ denote the classes in $H^*(\Hess(S,h))$ induced by $\mathcal T_u$ and $\mathcal T_{s_iu}$, respectively.


We recall the following proposition.

\begin{proposition}[{\cite[Theorem~B]{CHL}}]\label{prop:simple action}
	Let $w$ be an element in $\mathfrak{S}_n$ and let $s_i = s_{i,i+1}$ be a simple reflection.
\begin{enumerate}
	\item If $ w  \dashrightarrow s_iw$ or $s_iw \dasharrow w$, then $s_i \cdot \swh{w} =\swh{s_i w}$.
	\item If $s_i w \rightarrow w$, then $s_i \cdot \swh{w} = \swh{w}$.
	\item If $w \rightarrow s_iw$, then
	\[
	\left(s_i \cdot \swh{w} + \sum_{u \in \mathcal A_{s_i,w} } \tau_{s_iu} \right) = \swh{w} + \sum_{u   \in \mathcal A_{s_i,w} } \tau_u.
	\]
and the intersection $\mathcal A_{s_i,w} \cap s_i\mathcal A_{s_i,w}$ is empty. 
\end{enumerate}
\end{proposition}

\section{Module generators}\label{sec:module generators}

In view of Theorem~\ref{thm:Brosnan-Chow},
the number of permutation modules whose direct sum is $H^{2k}(\Hess(S,h))$ is expected to be the same as $m \colonequals  \dim_{\C} H^{2k}(\Hess(S,h))^{\mathfrak S_n}$.
In this section, we will show that there are $m$ classes $\sigma_{w,h}\in H^{2k}(\Hess(S,h))$, generating $H^{2k}(\Hess(S,h))$ as an $\mathfrak S_n$-module.

\begin{definition}\label{def:generators}
For  a Hessenberg function  $h \colon [n] \rightarrow [n]$,
   let
\[
\mathcal G_h \colonequals  \{ w \in \mathfrak{S}_n \mid  w^{-1}(w(j)+1) \leq h(j) \quad \text{ for } w(j) \in [n-1]\}
\]
and
\[
\mathcal G_h^k \colonequals  \{ w \in \mathcal G_h \mid \ell_h(w) = k \}.
\]

\end{definition}
In other words,
\begin{eqnarray*}
 \mathcal G_h &= & \mathfrak S_n \setminus \{ w \in \mathfrak S_n \mid w(j) +1 = w(i) \text{  for some } i >h(j)\} \\
  &=&   \{ w \in \mathfrak S_n\mid w(j) +1 \neq  w(i) \text{ for any } i >h(j)\}.
\end{eqnarray*}


\begin{proposition}[{cf. \cite[Lemma~2.3]{AHHM19}}] \label{prop:number of generators}
	For each $k \geq 0$, we have
	\[
	\lvert \mathcal G_h^k \rvert
= \dim_{\C} (H^{2k}(\Hess(S,h))^{\mathfrak{S}_n}).
\]
\end{proposition}
\begin{proof}
We first notice that Brosnan and Chow~\cite[Theorem~127]{BrosnanChow18} proved 
\[
 \dim_{\C} (H^{2k} (\Hess(S,h))^{\mathfrak{S}_n}) = \dim_{\C} H^{2k}(\Hess(N,h))
\]
for every $k \geq 0$. 	
Here, $N$ is the Jordan canonical form of a regular nilpotent element in~$\frak{gl}_n(\C)$:
\[ 
N = \left(
\begin{array}{ccccc}
0&1& & &\\
 &0&1& & \\
 & & \rotatebox{45}{\vdots} &\rotatebox{45}{\vdots} & \\
 & & & 0 &1 \\
 &&&& 0
 \end{array}
\right).
\]

Note that for each $w \in \mathfrak{S}_n$, the intersection of Schubert cell $\Xwo{w}$ and $\Hess(N,h)$ is nonempty if and only if
\begin{equation}\label{eq nonempty intersection}
w^{-1}(w(j)-1) \leq h(j) \quad \text{ for  }j \in [n]
\end{equation}
(see~\cite[Lemma~2.3]{AHHM19}). Here, we use the convention that $w(0) = 0$.
  Moreover, for such $w$, the dimension of the intersection is given by
\begin{equation}\label{eq_dim_intersection}
\dim_{\C}(\Hess(N,h) \cap \Xwo{w})
= \lvert \{ (j,i) \mid 1 \leq j < i \leq h(j), w(j) > w(i)\} \rvert
\end{equation}
(see \cite[Section 2.2]{AHT2020},  \cite[Theorem~35 and the remark after it]{BrosnanChow18}).

To complete the proof, it is enough to show that there is a bijective correspondence between the following two sets:
\begin{eqnarray}
&&\{ w \in \mathfrak{S}_n \mid \dim_{\C}\Owho{w} = k, w^{-1}(w(j)+1) \leq h(j) \quad \text{ for } w(j) \in [n-1]\},\label{eq_compare_set_1}\\
&&\{ w \in \mathfrak{S}_n \mid \dim_{\C}(\Hess(N,h) \cap \Xwo{w})  = k,w^{-1}(w(j)-1) \leq h(j) \quad \text{ for  }j \in [n] \}.\label{eq_compare_set_2}
\end{eqnarray}
Note that by the dimension formula in Proposition~\ref{prop_basics}(\ref{dim}), we obtain 
\[
\dim_{\C} \Owho{w}  = \lvert \{ (j,i) \mid 1 \leq j < i \leq h(j), w(j) < w(i)\} \rvert.
\]
We consider the involution $\iota \colon \mathfrak{S}_n \to \mathfrak{S}_n$ given by $(\iota(w))(i) = n-w(i) +1$. Then, we get $(\iota(w))^{-1}(i) = w^{-1}(n-i+1)$ and
\[
\begin{split}
(\iota(w))^{-1}((\iota(w))(j) + 1)
&= (\iota(w))^{-1}(n - w(j) +1 + 1)\\
&=(\iota(w))^{-1}(n - w(j) + 2) \\
&= (\iota(w))^{-1}(n - (w(j)-1) + 1) \\
&= w^{-1}(w(j)-1).
\end{split}
\]
Moreover, for $1 \leq j < i \leq h(j)$, we have  $w(j) > w(i)$ if and only if $\iota(w)(j) = n- w(j) + 1 < n-w(i) + 1 = \iota(w)(i)$.
Therefore, the involution gives a desired bijective correspondence between two sets in~\eqref{eq_compare_set_1} and~\eqref{eq_compare_set_2}.
This completes the proof.
\end{proof}

\begin{example}
	Suppose that $n = 4$ and $h = (2,4,4,4)$. Then the elements $w \in \mathfrak{S}_4$ satisfying the condition $w^{-1}(w(j) -1) \leq h(j)$ for $j \in [4]$ and their involutions $\iota(w)$ are given as follows.
	\smallskip
	\begin{center}
	\begin{tabular}{c|l||l}
		$\dim_{\C} (\Hess(N,h) \cap \Xwo{w})$ & $w$ &  $\iota(w)$ \\
		\hline
		$4$ & $4321$ & $1234$\\
		$3$ & $4312$, $3241$, $1432$ & $1243$, $2314$, $4123$\\
		$2$ & $3214$, $2143$, $1423$, $1342$ & $2341$, $3412$, $4132$, $4213$\\
		$1$ & $2134$, $1324$, $1243$ & $3421$, $4231$, $4312$\\
		$0$ & $1234$ & $4321$
	\end{tabular}
\end{center}
\end{example}

\begin{definition} \label{def_graph_Ghw}
	Let $h$ be a Hessenberg function.
\begin{enumerate}

\item The \emph{incomparability graph} $G_h$ of $h$ is the graph with the vertex set~$[n]$ and the edge set $\{\{j,i\}\mid j <i \leq h(j)\}$.

\item For $w \in \mathfrak{S}_n$,
	  define a directed graph~$G_{w,h}$ with the vertex set~$[n]$ such that for each pair of indices $1 \leq j < i \leq n$, there is an edge $j \to i$ in $G_{w,h}$ if and only if
	\begin{eqnarray*}
	j < i \leq h(j), \quad w(j) < w(i),
	\end{eqnarray*}
and define $\overline{G}_{w,h}$ by adding edges $j \leftarrow i$ to $G_{w,h}$ for any pair $(j,i)$ satisfying  $j<i\leq h(j)$ and $w(j) > w(i)$. Then $\overline{G}_{w,h}$ is the incomparability graph of $h$ with an acyclic orientation, denoted by  $o_h(w)$.

\item Denote by $\mathcal O_h$ the set of all acyclic orientations of the incomparability graph $G_h$ of $h$. For each $k$, define $\mathcal O_h^k$ by the set of all acyclic orientations of $G_h$ such that the number of edges $i\leftarrow j$ with $i<j$ is $k$. Then $\mathcal O_h = \bigsqcup_k \mathcal O_h^k$.

\item Define an equivalence relation $\sim_h$ on $\mathfrak S_n$ by   $v  \sim_h w$ if $G_{v,h} =G_{w,h}$. In this case, we say that $v$ and $w$ {\it have the same graph type}. Denote by $[w]_h$ the equivalence class containing $w$.

\end{enumerate}
\end{definition}

\begin{example}

Let $h = (2,4,4,4)$.
There are twelve different acyclic orientations on the incomparability graph $G_h$ of $h$, each of which is described by $\overline{G}_{w,h}$ for $w \in \mathcal G_h$. We describe this correspondence.
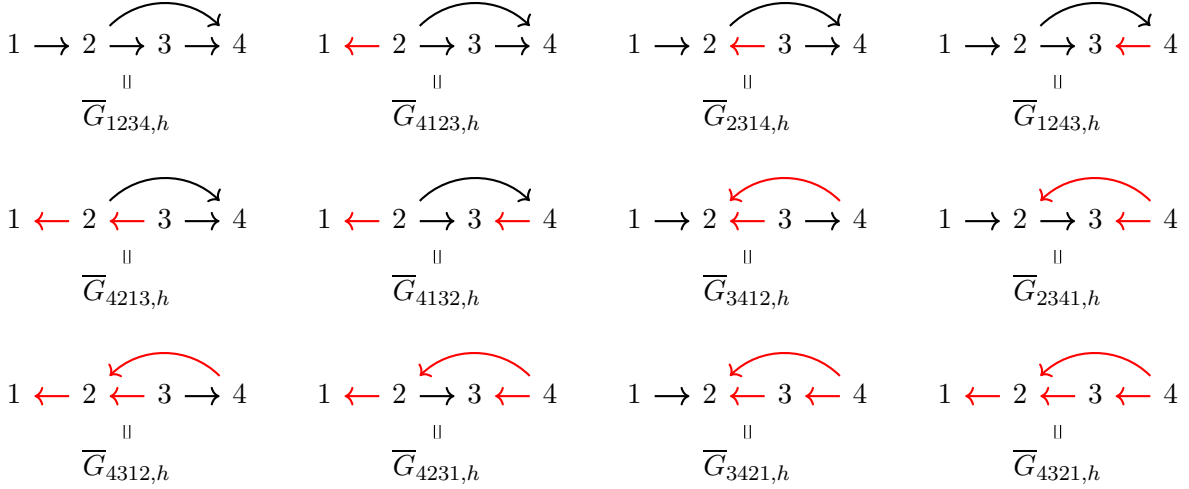
\[
\begin{tikzcd}[row sep = 0.5em, column sep = 0.7em]
\begin{tikzpicture}[scale = 1, every label/.append style={text=blue, font=\footnotesize}, baseline=-.5ex]

\node (1) at (1,0) {$1$};
\node  (2) at (2,0) {$2$};
\node (3) at (3,0) {$3$};
\node (4) at (4,0) {$4$};

\draw[->, thick] (1)--(2);
\draw[->, thick] (2)--(3);
\draw[->, thick] (3)--(4);
\draw[->, out=45, in=135, thick] (2) to (4);

\end{tikzpicture}
& \begin{tikzpicture}[scale = 1, every label/.append style={text=blue, font=\footnotesize}, baseline=-.5ex]

\node (1) at (1,0) {$1$};
\node  (2) at (2,0) {$2$};
\node (3) at (3,0) {$3$};
\node (4) at (4,0) {$4$};

\draw[<-, red, thick] (1)--(2);
\draw[->, thick] (2)--(3);
\draw[->, thick] (3)--(4);
\draw[->, out=45, in=135, thick] (2) to (4);

\end{tikzpicture}
& \begin{tikzpicture}[scale = 1, every label/.append style={text=blue, font=\footnotesize}, baseline=-.5ex]

\node (1) at (1,0) {$1$};
\node  (2) at (2,0) {$2$};
\node (3) at (3,0) {$3$};
\node (4) at (4,0) {$4$};

\draw[->, thick] (1)--(2);
\draw[<-, red, thick] (2)--(3);
\draw[->, thick] (3)--(4);
\draw[->, out=45, in=135, thick] (2) to (4);

\end{tikzpicture}
&\begin{tikzpicture}[scale = 1, every label/.append style={text=blue, font=\footnotesize}, baseline=-.5ex]

\node (1) at (1,0) {$1$};
\node  (2) at (2,0) {$2$};
\node (3) at (3,0) {$3$};
\node (4) at (4,0) {$4$};

\draw[->, thick] (1)--(2);
\draw[->, thick] (2)--(3);
\draw[<-, red, thick] (3)--(4);
\draw[->, out=45, in=135, thick] (2) to (4);

\end{tikzpicture}  \\
\overline{G}_{1234,h} \arrow[u,equal]
& \overline{G}_{4123,h} \arrow[u,equal]
& \overline{G}_{2314,h} \arrow[u,equal]
& \overline{G}_{1243,h} \arrow[u,equal] \\
\begin{tikzpicture}[scale = 1, every label/.append style={text=blue, font=\footnotesize}, baseline=-.5ex]

\node (1) at (1,0) {$1$};
\node (2) at (2,0) {$2$};
\node (3) at (3,0) {$3$};
\node (4) at (4,0) {$4$};

\draw[<-, red, thick] (1)--(2);
\draw[<-,red, thick] (2)--(3);
\draw[->, thick] (3)--(4);
\draw[->, out=45, in=135, thick] (2) to (4);

\end{tikzpicture}
& \begin{tikzpicture}[scale = 1, every label/.append style={text=blue, font=\footnotesize}, baseline=-.5ex]

\node (1) at (1,0) {$1$};
\node (2) at (2,0) {$2$};
\node (3) at (3,0) {$3$};
\node (4) at (4,0) {$4$};

\draw[<-, red, thick] (1)--(2);
\draw[->, thick] (2)--(3);
\draw[<-, red, thick] (3)--(4);
\draw[->, out=45, in=135, thick] (2) to (4);

\end{tikzpicture}
& \begin{tikzpicture}[scale = 1, every label/.append style={text=blue, font=\footnotesize}, baseline=-.5ex]

\node (1) at (1,0) {$1$};
\node (2) at (2,0) {$2$};
\node (3) at (3,0) {$3$};
\node (4) at (4,0) {$4$};

\draw[->, thick] (1)--(2);
\draw[<-, red, thick] (2)--(3);
\draw[->, thick] (3)--(4);
\draw[<-,red, out=45, in=135, thick] (2) to (4);

\end{tikzpicture}
&\begin{tikzpicture}[scale = 1, every label/.append style={text=blue, font=\footnotesize}, baseline=-.5ex]

\node (1) at (1,0) {$1$};
\node  (2) at (2,0) {$2$};
\node (3) at (3,0) {$3$};
\node (4) at (4,0) {$4$};

\draw[->, thick] (1)--(2);
\draw[->, thick] (2)--(3);
\draw[<-, red, thick] (3)--(4);
\draw[<-,red, out=45, in=135, thick] (2) to (4);

\end{tikzpicture}  \\
\overline{G}_{4213,h} \arrow[u,equal]
& \overline{G}_{4132,h} \arrow[u,equal]
& \overline{G}_{3412,h} \arrow[u,equal]
& \overline{G}_{2341,h} \arrow[u,equal] \\
\begin{tikzpicture}[scale = 1, every label/.append style={text=blue, font=\footnotesize}, baseline=-.5ex]

\node (1) at (1,0) {$1$};
\node (2) at (2,0) {$2$};
\node (3) at (3,0) {$3$};
\node (4) at (4,0) {$4$};

\draw[<-, red, thick] (1)--(2);
\draw[<-,red, thick] (2)--(3);
\draw[->, thick] (3)--(4);
\draw[<-, red, out=45, in=135, thick] (2) to (4);

\end{tikzpicture}
& \begin{tikzpicture}[scale = 1, every label/.append style={text=blue, font=\footnotesize}, baseline=-.5ex]

\node (1) at (1,0) {$1$};
\node (2) at (2,0) {$2$};
\node (3) at (3,0) {$3$};
\node (4) at (4,0) {$4$};

\draw[<-, red, thick] (1)--(2);
\draw[->, thick] (2)--(3);
\draw[<-, red, thick] (3)--(4);
\draw[<-,red, out=45, in=135, thick] (2) to (4);

\end{tikzpicture}
& \begin{tikzpicture}[scale = 1, every label/.append style={text=blue, font=\footnotesize}, baseline=-.5ex]

\node (1) at (1,0) {$1$};
\node (2) at (2,0) {$2$};
\node (3) at (3,0) {$3$};
\node (4) at (4,0) {$4$};

\draw[->, thick] (1)--(2);
\draw[<-, red, thick] (2)--(3);
\draw[<-,red, thick] (3)--(4);
\draw[<-,red, out=45, in=135, thick] (2) to (4);

\end{tikzpicture}
&\begin{tikzpicture}[scale = 1, every label/.append style={text=blue, font=\footnotesize}, baseline=-.5ex]

\node (1) at (1,0) {$1$};
\node  (2) at (2,0) {$2$};
\node (3) at (3,0) {$3$};
\node (4) at (4,0) {$4$};

\draw[<-,red, thick] (1)--(2);
\draw[<-,red, thick] (2)--(3);
\draw[<-, red, thick] (3)--(4);
\draw[<-,red, out=45, in=135, thick] (2) to (4);

\end{tikzpicture}  \\
\overline{G}_{4312,h} \arrow[u,equal]
& \overline{G}_{4231,h} \arrow[u,equal]
& \overline{G}_{3421,h} \arrow[u,equal]
& \overline{G}_{4321,h} \arrow[u,equal] \\
\end{tikzcd}
\]
\end{example}

\begin{proposition} \label{prop:generator and acyclic orientation} Let $h\colon[n] \rightarrow [n]$ be a Hessenberg function. Then, we have
\[
\lvert \mathcal G_h ^k \rvert= \lvert \mathcal{O}_h ^k\rvert. 
\]
\end{proposition}

\begin{proof}
For each $\lambda \vdash n$, let $c_{\lambda}^h(t)$ be the coefficient of the elementary symmetric function $e_{\lambda}$ in the $e$-basis expansion of the chromatic quasisymmetric function $X_{G_h}({\bf x},t)$ of the incomparability graph $G_h$ of $h$.
By Theorem~5.3 of \cite{SW},
\[
\sum_{\lambda \in Par(n,j)}c_{\lambda}^h(t) = \sum_{o \in \mathcal O(G_h,j)} t^{\mathrm{asc}(o)},
\]
where
$Par(n,j)$ is the set of partitions of $n$ of length $j$,
$\mathcal O(G_h,j)$ is the set of acyclic orientations of $G_h$ with $j$ sinks, and $\mathrm{asc}(o)$ is the number of directed edges $(a,b)$ of $o$ for which~$a <b$.

On the other hand, by Theorem~\ref{thm:Brosnan-Chow}, we have
 $\sum_{\lambda \vdash n} c_{\lambda}^h(t) M^{\lambda} = \sum_k t^k H^{2k}(\Hess(S,h))$  in the ring of $\mathfrak S_n$-representations,
 from which it follows that
\[
\sum_{\lambda \vdash n}  c_{\lambda}^h(t) =\sum_k t^k\dim_{\C} H^{2t}(\Hess (S,h))^{\mathfrak S_n}
\]
   because $\dim_{\C} (M^{\lambda})^{\mathfrak S_n}=1$ for any permutation module $M^{\lambda}$.
By Proposition \ref{prop:number of generators},
the right-hand side is equal to $\sum_k t^k \lvert \mathcal G_h^k\rvert$.
 Therefore, $\lvert \mathcal O_h^k \rvert= \lvert \mathcal G_h^k \rvert$.
\end{proof}

We recall the following lemma related to the graph type.
\begin{lemma}[{\cite[Lemma~4.4]{CHL}}]\label{lemma_graph_type}
Let $v$ be a permutation in $\mathfrak{S}_n$ and let $s_i$ be a simple reflection.
If $v \dasharrow s_iv$, then $G_{v,h} = G_{s_i v,h}$.
\end{lemma}

For two permutations $v,w \in \mathfrak{S}_n$ such that $w = s_i v$ and $v \dasharrow w$, we denote by $v \stackrel{s_i}{\dasharrow} w$.

 \begin{proposition} \label{prop:type and orbit pre}
 For any $u \in \mathfrak S_n$, there exists a unique $w \in \mathcal G_h$ such that $G_{u,h} = G_{w,h}$. In this case,  we have
\[
w \stackrel{s_{i_1}}{\dasharrow} \cdots \stackrel{s_{i_r}}{\dasharrow} u
\]
and   $\sigma_{u,h} = s_{i_r} \cdots s_{i_1} \cdot \sigma_{w,h}$ for some simple reflections $s_{i_1}, \dots, s_{i_r}$.
 \end{proposition}

\begin{proof}
We claim that a maximal element with respect to the Bruhat order in each equivalence class $[ \,\,\,]_h$ is an element of $\mathcal G_h$.
Let $v \in \mathfrak S_n$ which is not an element of $\mathcal G_h$. Then there is $j<k$ with $v(j)=i$ and $v(k)=i+1$ and the vertex $j$ is not connected by an edge to the vertex~$k$ in $G_{v,h}$. Thus we have $s_iv \dasharrow v$     and $G_{s_iv,h} = G_{ v,h}$ by Lemma~\ref{lemma_graph_type}.  Therefore,  $v$ is not a maximal element with respect to Bruhat order in its equivalence class.
This completes the proof of the claim. By the claim, the map $\mathcal G_h \rightarrow \mathfrak S_n/\sim_h$ assigning $w \in \mathcal G_h$ to its equivalent class $[w]_h$ is surjective.

On the other hand, an equivalence class $[ w]_h$ induces an element $o_h(w)$ in $\mathcal O_h$
(see Definition~\ref{def_graph_Ghw}(2)).
By  a similar argument as in the proof of Proposition 4.1 of \cite{CH}, 
the map
\[
[w]_h \in \mathfrak S_n/\sim_h \,\, \mapsto \,\, o_h(w) \in \mathcal O_h
\]
  is surjective, whereas provide a proof for the convenience of the reader.

Given an acyclic orientation $o  $ on $G_h$,  we  define a directed graph $\Gamma_{\ell}$ with $ n-\ell+1  $ vertices and a vertex $\s_{\ell}  $ of $\Gamma_{\ell}$ for $\ell \in [n]$ inductively, as follows. For $\ell=1$, let $\Gamma_1$ be the graph $G_h$ with the acyclic orientation $o$ and define $\s_1$ by the maximal source of $o$.
Assume that we have defined $\Gamma_{\ell-1}$ and $\s_{\ell-1}$. Define $\Gamma_{\ell}$ by the directed graph obtained from $\Gamma_{\ell-1}$ by deleting $\s_{\ell-1}$ and every edge from $\s_{\ell-1}$, and define $\s_{\ell}$ by the maximal source of $\Gamma_{\ell}$.
Let $w \in \mathfrak S_n$ be  the permutation  defined by $w(\s_{\ell})=\ell$ for $\ell \in [n]$.

We now claim that $o_h(w) = o$. It is enough to show that $o_h(w)$ and $o$ have the same set of  directed edges $j \rightarrow i$ with $j<i$ to prove the claim.   To see this, let $\s_{\ell}$ and $\s_k$ ($\ell<k$) be  two vertices  which are connected by an edge in $G_h$. Then this edge is directed as  $\s_{\ell} \rightarrow \s_k$ in $o$ because $\s_{\ell}$ is a source in the directed graph $\Gamma_{\ell}$ which has $\s_{k}$ as a vertex.
Since we have $w(\s_{\ell})=\ell  < w(\s_{k})=k$, the edge $\{\s_{\ell}, \s_{k}\}$ is directed from $s_{\ell}$ to $ \s_{k}$ in $G_{w,h}$.   Therefore, $o_h(w)$ and $o$ have the same set of  directed edges $j \rightarrow i$ with $j<i$.
Consequently,  we get $o_h(w) =o$.

Therefore, the composition $\mathcal G_h \rightarrow \mathfrak S_n/\sim_h \rightarrow \mathcal O_h$ defines   a surjective map
$ \mathcal G_h \rightarrow   \mathcal O_h$.
By Proposition~\ref{prop:generator and acyclic orientation},   we have  $\lvert \mathcal G_h \rvert =\lvert \mathcal O_h \rvert$.
Consequently,
the map $\mathcal G_h \rightarrow \mathfrak S_n/\sim_h$ is injective.
It follows that   each equivalence class $[ \,\,\,]_h$ has a unique maximal element.

 Furthermore, the proof of the claim implies that for any element $u$ in the same equivalence class as $w \in \mathcal G_h$, there exist simple reflections $s_{i_1}, \dots, s_{i_r}$ such that $w \stackrel{s_{i_1}}{\dasharrow} \cdots \stackrel{s_{i_r}}{\dasharrow} u$.  Then, by Proposition \ref{prop:simple action}(1), we get $\sigma_{u,h} = s_{i_r} \cdots s_{i_1} \cdot \sigma_{w,h}$.
\end{proof}

\begin{remark}
In the proof of Proposition \ref{prop:type and orbit pre},  for a given   acyclic orientation $o$ on $G_h$, we assign a permutation $w$ with $o_h(w)=o$. In fact, such a permutation $w$  is an element of $\mathcal G_h$: Suppose that $\s_{\ell-1} < \s_{\ell}$ and there is no edge connecting $\s_{\ell-1}$ and $\s_{\ell}$ in $G_h$. Then $\s_{\ell}$ is a source in $\Gamma_{\ell-1}$ because  $\s_{\ell}$ is a source in~$\Gamma_{\ell}$ which is obtained from $\Gamma_{\ell-1}$ by removing $\s_{\ell-1}$ and all edges from~$\s_{\ell-1}$, contradicting to the condition that $\s_{\ell-1}$ is the maximal source of the directed graph~$\Gamma_{\ell-1}$. Therefore, if $\s_{\ell-1} < \s_{\ell}$, then there is an edge connecting $\s_{\ell-1}$ and $\s_{\ell}$ in $G_h$. In this case, the orientation is from $\s_{\ell-1}$ to $\s_{\ell}$ because $w(s_{\ell-1})=\ell-1 < w(\s_{\ell}) =\ell$.
\end{remark}

   For $w \in \mathcal G_h^k$, define a subset $P_{w,h}$ of $\mathfrak S_n$ by
\[
P_{w,h} \colonequals  \{ u \in \mathfrak S_n \mid G_{u,h} = G_{w,h}\} = [w]_h
\]
and  denote by $M(\sigma_{w,h})$ the $\mathfrak S_n$-module generated by $\sigma_{w,h} \in H^{2k}(\Hess(S,h))$. Then Proposition \ref{prop:type and orbit pre} can be rephrased as follows.

\begin{proposition} \label{prop:type and orbit}
For each $w \in \mathcal G_h^k$,  any element in $\{\sigma_{u,h} \mid u \in P_{w,h}\}$ is contained in~$M(\sigma_{w,h})$, and $\{P_{w,h} \mid w \in \mathcal G_h^k\}$  defines a partition on the set $\{\sigma_{u,h} \mid  u\in \mathfrak S_n, \ell_h(u) = k \}$, that is,
\[
\{ \sigma_{u,h} \mid u\in \mathfrak S_n, \ell_h(u) =k \} = \bigsqcup_{w \in \mathcal G_h^k} \{ \sigma_{u,h} \mid u \in P_{w,h}\}.
\]
\end{proposition}

\begin{theorem} \label{thm:module generator}
The set $\{\sigma_{w,h} \mid w \in \mathcal G_h^k\}$ has the cardinality $\dim_{\C} H^{2k}(\Hess(S,h))^{\mathfrak S_n}$, and it generates the  $\mathfrak S_n$-module  $H^{2k}(\Hess(S,h))$, that is,
\[
H^{2k}(\Hess(S,h)) = \sum_{w \in \mathcal G_h^k} M(\sigma_{w,h}),
\]
where $M(\sigma_{w,h})$ denotes the $\mathfrak S_n$-module generated by $\sigma_{w,h}$.
\end{theorem}

\begin{proof}  The first statement follows from Proposition  \ref{prop:number of generators}.
By Proposition \ref{prop:type and orbit}, for any $u \in \mathfrak S_n$, there is $w \in \mathcal G_h$ with $\sigma_{u,h} \in M(\sigma_{w,h})$.
Since $\{\sigma_{u,h} \mid u \in \mathfrak S_n\}$ is a  $\mathbb C$-basis of $H^*(\Hess(S,h))$, the set $\{\sigma_{w,h} \mid w \in \mathcal G_h\}$ generates the  $\mathfrak S_n$-module  $H^*(\Hess(S,h))$.
\end{proof}

We recall the definition of permutation modules of the symmetric group $\mathfrak S_n$.
A \emph{composition}~$\alpha=(\alpha_1, \dots, \alpha_\ell)$ of $n$ is a
sequence of positive integers such that $\sum_i^\ell \alpha_i =n$. For a
composition~$\alpha=(\alpha_1, \dots, \alpha_\ell)$ of $n$, we let $\mathfrak S_\alpha$ be the \emph{Young subgroup} of $\mathfrak S_n$ defined as
\[
\mathfrak S_\alpha=\mathfrak S_{\{1, \dots, \alpha_1 \}}\times\mathfrak
S_{\{\alpha_1+1, \dots, \alpha_1+\alpha_2 \}}\times \cdots \times \mathfrak
S_{\{n-\alpha_\ell+1, \dots, n\}},
\]
and let $M^\alpha$ be the \emph{permutation module} of $\mathfrak S_n$ associated to $\alpha$ defined as the induced module~$1 \!\! \uparrow_{\mathfrak S_\alpha}^{\mathfrak S_n}$. Indeed, $M^\alpha$ is isomorphic to $M^\lambda$ for a partition $\lambda$ obtained by rearranging the parts of~$\alpha$ in nonincreasing order.

We remark that the sum $\sum_{w \in \mathcal G_h^k}  M(\sigma_{w,h})$ is generally not a direct sum, and the question is whether we can reduce modules $M(\sigma_{w,h})$ so that we get a direct sum decomposition into permutation modules, as in Conjecture \ref{conj:permutation module decomposition_intro}.  A natural question  related to this conjecture is the following.

\begin{question}\label{question}
Let $w \in \mathcal G_h$.
\begin{enumerate}
\item Is the stabilizer $\Stab_{\mathfrak S_n}(\sigma_{w,h})$ of $\sigma_{w,h}$ in $\mathfrak S_n$ a Young subgroup?
\item Is the $\mathfrak S_n$-module $M(\sigma_{w,h})$ generated by $\sigma_{w,h}$ a permutation module?
\end{enumerate}
\end{question}

For $w \in \mathfrak{S}_n$,
let ${\bf J}_{w,h}$ be the subset of $[n-1]$
consisting of $i \in [n-1]$ 
such that $w \dasharrow s_i w$, $s_iw \dasharrow w$, or
\[
w \rightarrow s_iw \quad \text{ and } \quad \{ u \in \Omega_{s_iw,h}^T \cap \Omega_{w}^T \mid u \dasharrow s_iu   \text{ and }\dim_{\C} (\Omega_u^{\circ} \cap \Omega_{s_iw,h}) = \dim_{\C} \Omega_{w,h}\}\neq \emptyset.
\]
Then    the Young subgroup 
\[
\mathfrak S_{{\alpha}_{w}} = 
\langle s_i \mid i \notin {\bf J}_{w,h} \rangle
\] 
stabilizes $\sigma_{w,h}$ by Proposition~\ref{prop:simple action}. We expect that $\mathfrak S_{{\alpha}_{w}}$ is   indeed the stabilizer $\Stab_{\mathfrak S_n}(\sigma_{w,h})$, proving that  $\Stab_{\mathfrak S_n}(\sigma_{w,h})$ is a Young subgroup.

{Question~\ref{question}}(2) is trickier than {Question~\ref{question}}(1) as the following example shows.

\begin{example}
Let $h=(2,4,4,4)$ and $w=3412$. Then ${\bf J}_{w,h}$ is $\{2\}$ and the stabilizer $\Stab_{\mathfrak S_4}(\sigma_{w,h})$ is the Young subgroup $\mathfrak S_{(2,2)}$ but the $\mathfrak S_n$-module $M(\sigma_{w,h})$ generated by $\sigma_{w,h}$ is the permutation module $M^{(3,1)}$ of type $(3,1)$. More precisely, the set $\{u \cdot \sigma_{w,h} \mid u \in \mathfrak S_4\}$ is given by

\[
		\begin{tikzcd}[column sep = 1em]
		\sigma_{3412,h} \arrow[r, "s_2"]
		& \sigma_{2413,h}  \arrow[r, "s_1"] \arrow[d, "s_3"]
		& \sigma_{1423,h}   \arrow[d, "s_3"] &\\
		& 
		\sigma_{2413,h}
		\arrow[r, "s_1"]
		& \sigma_{1423,h} \arrow[r, "s_2"]
		&\sigma _{1423,h} \\ [-2em]
		& +  (\sigma_{4213,h}  -  \sigma_{3214,h})  & +(\sigma_{4213,h} -\sigma_{3214,h})  &
+ \sigma_{4213,h}   -(\sigma_{3214,h} +\sigma_{3412,h} -\sigma_{2413,h}). 
		\end{tikzcd}
		\]
From this, we see that six elements
\[
\sigma_{w,h}, \quad s_2 \cdot\sigma_{w,h}, \quad s_1s_2  \cdot \sigma_{w,h},\quad  s_3s_2  \cdot \sigma_{w,h},\quad s_3s_1s_2 \cdot \sigma_{w,h},\quad s_2 s_3s_1s_2  \cdot \sigma_{w,h}
\] 
are all distinct but they span a $4$-dimensional vector space  in $H^4(\Hess(S,h))$.
Indeed, there are two linear relations:
\[
s_3s_1s_2  \cdot \swh{w} = s_3s_2  \cdot \swh{w} - s_2  \cdot \swh{w} + s_1s_2  \cdot \swh{w}, \quad
s_2s_3s_1s_2 \cdot  \swh{w}
= - \swh{w} + s_3s_2 \cdot \swh{w} + s_1s_2 \cdot \swh{w}.
\]

\end{example}

We will give an affirmative answer to Question~\ref{question} when $w \in \mathcal G_h^1$ in Propositions~\ref{prop:stabilizer_12},~\ref{prop:stabilizer_34}, and Theorem~\ref{thm_H2}. Also, see Remark~\ref{rmk_question}.



\section{\texorpdfstring{$\mathfrak{S}_n$}{Sn}-action on generators}\label{sec:generators of H2}
Recall from Definition~\ref{def:generators} and Theorem~\ref{thm:module generator} that the  set $\mathcal{G}_h^1$ of permutations whose corresponding classes form a generator set of $H^2(\Hess(S,h))$ is given as follows
\[ \mathcal{G}_h^1 =  \{w\in \mathfrak S_n \mid \ell_h(w)=1 \text{ and } w^{-1}(w(j)+1)\leq h(j) \text{ for } w(j)\in [n-1] \}.\]
In this section, we determine all the permutations in $\mathcal{G}_h^1$ and consider the stabilizer subgroups of their corresponding classes.
We first define a permutation $\wi{i}\in \mathcal{G}_h^1$ for each $i=1, \dots, n-1$. Before introducing the definition, we prepare one terminology.
\begin{definition}\label{def_of_T}
For a Hessenberg function $h \colon [n] \to [n]$, we define
\[
T = T_h \colonequals  \{i \in [n-1] \mid h(i-1) > (i-1)+1 = i\} \subset [n-1].
\]
Here, we set $h(0) = 2$.
\end{definition}
We note that $1 \in T$ always holds because $h(0) = 2 > 1$.
\begin{remark}
For each element $i$ in $T$, we construct a \emph{trivial}
representation (see Theorems~\ref{thm:Chow} and~\ref{thm_H2}).
The notation $T$ stands for the word `trivial'.
\end{remark}

For a permutation $w \in \mathfrak{S}_n$, we say that $i \in [n-1]$ is a \emph{descent} of $w$ if $w(i) > w(i+1)$. Let $\Des(w)$ be the set of descents of $w$, and let $\des(w) = |\Des(w)|$ be the number of descents of~$w$.

\begin{definition}\label{def:permutations}
Let $h \colon [n] \rightarrow [n]$ be a Hessenberg function. We define $\wi{i} \in\mathfrak S_n $ for $i=1, 2, \dots, n-1$ according to the values of $h$ as follows. We place $\vred$ between the $i$th and $(i+1)$st values of $\wi{i}$, and provide the graph $\overline{G}_{\wi{i},h}$ below the definition of $\wi{i}$ for each case.
Let $T = T_h$.

\begin{enumerate}
\item Suppose that $i \in T$.
\begin{enumerate}
\item If $i+1 \in T \cup \{n\}$, then we let
\[ \wi{i} \colonequals 1\, 2\,\cdots i-1 \,\,i+1\, \vred \, i \,\, i+2\,\cdots\,\, n = s_i.
\]
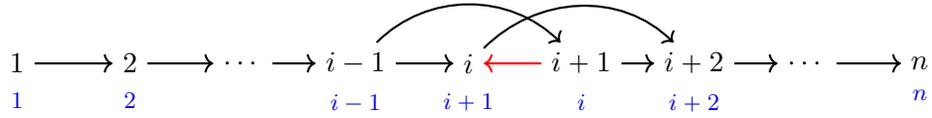
\begin{center}
\begin{tikzpicture}[scale = 1.5, every label/.append style={text=blue, font=\footnotesize}]
\tikzstyle{v} = [ ]
\node[v, label={below:{$1$}}] (1) at (1,0) {$1$};
\node[v, label={below:{$2$}}] (2) at (2,0) {$2$};
\node (3) at (3,0) {$\cdots$};
\node[v, label={below:{$i-1$}}] (4) at (4,0) { $i-1$};
\node[v, label={below:{$i+1$}}] (5) at (5,0) {$i$};
\node[v, label={below:{$i$}}] (6) at (6,0) { $i+1$};
\node[v, label={below:{$i+2$}}] (7) at (7,0) { $i+2$};
\node (8) at (8,0) {$\cdots$};
\node[v, label={below:{$n$}}] (9) at (9,0) {$n$};

\draw[->, thick] (1) to (2);
\draw[->, thick] (2) to (3);
\draw[->, thick] (3) to (4);
\draw[->, thick] (4) to (5);
\draw[<-, red, thick] (5) to (6);
\draw[->, thick] (6) to (7);
\draw[->, thick] (7) to (8);
\draw[->, thick] (8) to (9);

\draw[->, out=45, in=135, thick] (5) to (7);
\draw[->, out=45, in=135, thick] (4) to (6);
\end{tikzpicture}	

\end{center}
\item If $i+1 \notin T \cup \{n\}$, then we let
\[
\wi{i} \colonequals  1\,2\,\cdots i-1\,\, n \,\vred \, i\,\, i+1 \, \cdots \, n-1.
\]
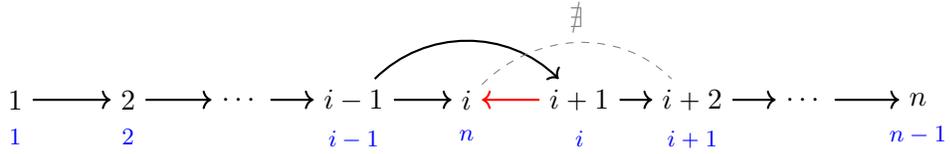
\begin{center}
\begin{tikzpicture}[scale = 1.5, every label/.append style={text=blue, font=\footnotesize}]
\tikzstyle{v} = [ ]
\node[v, label={below:{$1$}}] (1) at (1,0) {$1$};
\node[v, label={below:{$2$}}] (2) at (2,0) {$2$};
\node (3) at (3,0) {$\cdots$};
\node[v, label={below:{$i-1$}}] (4) at (4,0) { $i-1$};
\node[v, label={below:{$n$}}] (5) at (5,0) {$i$};
\node[v, label={below:{$i$}}] (6) at (6,0) { $i+1$};
\node[v, label={below:{$i+1$}}] (7) at (7,0) {$i+2$};
\node (8) at (8,0) {$\cdots$};
\node[v, label={below:{$n-1$}}] (9) at (9,0) {$n$};

\draw[->, thick] (1) to (2);
\draw[->, thick] (2) to (3);
\draw[->, thick] (3) to (4);
\draw[->, thick] (4) to (5);
\draw[<-, red, thick] (5) to (6);
\draw[->, thick] (6) to (7);
\draw[->, thick] (7) to (8);
\draw[->, thick] (8) to (9);

\draw[->, out=45, in=135, thick] (4) to (6);
\draw[ out=45, in=135, gray, dashed] (5) to node[midway, above] {$\not\exists$} (7);
\end{tikzpicture}	
\end{center}
\end{enumerate}

\item Suppose that $i \notin T$.
\begin{enumerate}
\item If $i+1 \in T \cup \{n\}$, then we let
\[ \wi{i} \colonequals  2\,\cdots \,i \,\,i+1\,\vred\, 1\,\,i+2 \, \cdots \, n \,\,.\]

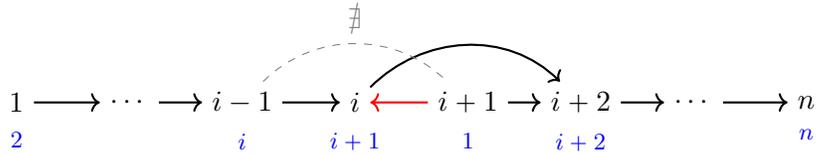
\begin{center}
\begin{tikzpicture}[scale = 1.5, every label/.append style={text=blue, font=\footnotesize}]
\tikzstyle{v} = [ ]
\node[v, label={below:{$2$}}] (1) at (1,0) {$1$};
\node (2) at (2,0) {$\cdots$};
\node[v, label={below:{$i$}}] (3) at (3,0) { $i-1$};
\node[v, label={below:{$i+1$}}] (4) at (4,0) {$i$};
\node[v, label={below:{$1$}}] (5) at (5,0) { $i+1$};
\node[v, label={below:{$i+2$}}] (6) at (6,0) { $i+2$};
\node (7) at (7,0) {$\cdots$};
\node[v, label={below:{$n$}}] (8) at (8,0) {$n$};

\draw[->, thick] (1) to (2);
\draw[->, thick] (2) to (3);
\draw[->, thick] (3) to (4);
\draw[<-, red, thick] (4) to (5);
\draw[->, thick] (5) to (6);
\draw[->, thick] (6) to (7);
\draw[->, thick] (7) to (8);

\draw[->, out=45, in=135, thick] (4) to (6);

\draw[ out=45, in=135, gray, dashed] (3) to node[midway, above] {$\not\exists$} (5);

\end{tikzpicture}	

\begin{tikzpicture}[scale = 1.5, every label/.append style={text=blue, font=\footnotesize}]
\tikzstyle{v} = [ ]

\node[v, label={below:{$2$}}] (1) at (1,0) {$1$};
\node (2) at (2,0) {$\cdots$};
\node[v, label={below:{$n-1$}}] (3) at (3,0) { $n-2$};
\node[v, label={below:{$n$}}] (4) at (4,0) {$n-1$};
\node[v, label={below:{$1$}}] (5) at (5,0) { $n$};

\draw[->, thick] (1) to (2);
\draw[->, thick] (2) to (3);
\draw[->, thick] (3) to (4);
\draw[<-, red, thick] (4) to (5);

\draw[ out=45, in=135, gray, dashed] (3) to node[midway, above] {$\not\exists$} (5);

\end{tikzpicture}	

\end{center}

\item If $i+1 \notin T \cup \{n\}$, then we let
\[ \wi{i} \colonequals  n-i+1\,\cdots \,n-1\,\,n\,\vred\, 1\,\, 2 \, \cdots \, n-i \,\,.\]
\begin{center}
\begin{tikzpicture}[scale = 1.5, every label/.append style={text=blue, font=\footnotesize}]
\tikzstyle{v} = [ ]

\node[v, label={below:{$n-i+1$}}] (1) at (1,0) {$1$};
\node (2) at (2,0) {$\cdots$};
\node[v, label={below:{$n-1$}}] (3) at (3,0) { $i-1$};
\node[v, label={below:{$n$}}] (4) at (4,0) {$i$};
\node[v, label={below:{$1$}}] (5) at (5,0) { $i+1$};
\node[v, label={below:{$2$}}] (6) at (6,0) { $i+2$};
\node (7) at (7,0) {$\cdots$};
\node[v, label={below:{$n-i$}}] (8) at (8,0) {$n$};

\draw[->, thick] (1) to (2);
\draw[->, thick] (2) to (3);
\draw[->, thick] (3) to (4);
\draw[<-, red, thick] (4) to (5);
\draw[->, thick] (5) to (6);
\draw[->, thick] (6) to (7);
\draw[->, thick] (7) to (8);

\draw[ out=45, in=135, gray, dashed] (4) to node[midway, above] {$\not\exists$} (6);

\draw[ out=45, in=135, gray, dashed] (3) to node[midway, above] {$\not\exists$} (5);

\end{tikzpicture}

\end{center}
\end{enumerate}
\end{enumerate}
Here, for each vertex $j$ in the graph $\overline{G}_{\wi{i},h}$, we write the value $\wi{i}(j)$ in blue below~$j$.
\end{definition}

\begin{lemma}\label{lemma:acyclic orientation_in_Gh1} The number of acyclic orientations that have exactly one directed edge from a vertex~$j$ to a vertex $i$ with $i<j$ is at most $n-1$.
\end{lemma}
\begin{proof}
If $j \rightarrow i$ with $i+1<j$ is an edge in an orientation of $G_h$, then the directed edges
\[
i\rightarrow i+1,\quad \dots, \quad
j-1 \rightarrow j,\quad j \rightarrow i
\]
form a cycle. Therefore, a unique edge from $j$ to $i$ where $i<j$ must occur only when $j=i+1$.
\end{proof}

It is easy to check that the permutations $\wi{i}$, $i=1,  \dots, n-1$, are all distinct and contained in~$\mathcal{G}_h^1$.
We have $\lvert\mathcal G_h^1 \rvert = \lvert \mathcal O_h^1 \rvert$ by Proposition~\ref{prop:generator and acyclic orientation} and combining this with the previous lemma, we obtain the following proposition.

\begin{proposition}\label{prop:wi_elements_in_Gh1} For a Hessenberg function $h \colon [n] \to [n]$, we have
	\[
	\Ghone = \{ \wi{i} \mid i \in [n-1]\}.
	\]
\end{proposition}

We now consider the $\mathfrak S_n$-action on the  generators $\swh{\wi{i}}$ of $H^2(\Hess(S,h))$ to analyze their stabilizer subgroup. Recall that we can use Proposition~\ref{prop:simple action} to compute the action of simple reflections, while we need to understand the set $\mathcal{A}_{s_i,w}$ when $w=s_i$ for the explicit computation.

We denote by $\mathcal{A}_i$ the set $\mathcal A_{s_i,s_i}$ of permutations that is used to describe
the  class $s_i \cdot \swh{s_i}$ in Proposition~\ref{prop:simple action}.
By the definition of $\mathcal{T}_u$, we have
\begin{equation}\label{eq_Tu}
\mathcal{T}_u = \overline{\Owo{u} \cap \Owh{s_i \cdot s_i}} =
\overline{\Owo{u} \cap \Owh{e}} =
\overline{\Owo{u} \cap \Hess(S,h)} = \Owh{u}.
\end{equation}
Here, $e$ is the identity element in $\mathfrak{S}_n$.
Accordingly, we get
\begin{equation}\label{eq_si_action_modified}
 \left(s_i\cdot \swh{s_i}  + \sum_{u \in \mathcal{A}_i } \swh{s_iu}  \right) =
\swh{s_i}  + \sum_{u \in \mathcal{A}_i } \swh{u}. 
\end{equation}
Furthermore, the intersection $\mathcal A_i  \cap s_i\mathcal A_i $ is empty by Proposition~\ref{prop:simple action}.
As a direct consequence of~\eqref{eq_si_action_modified}, we obtain the following
lemma.
\begin{lemma} \label{lem: empty implies stabilize}
	For $1 \leq i \leq n-1$, the set $\mathcal{A}_i$ is empty if and only if
	$s_i \cdot \swh{s_i} = \swh{s_i}$,  that is, the $s_i$-action stabilizes the
	class $\swh{s_i}$.
\end{lemma}

\begin{proof} By~\eqref{eq_si_action_modified}, if the set $\mathcal{A}_i $ is empty, then   we obtain $s_i \cdot \swh{s_i} = \swh{s_i}$.
Suppose that the set~$\mathcal A_i$ is nonempty. Since the intersection  $\mathcal A_i  \cap s_i \mathcal A_i $ is  empty and the classes $\{ \swh{u}\}$ form a basis by Proposition \ref{prop_basis}, the sum
\[
  \sum_{u \in \mathcal{A}_i } \swh{s_iu}  -
  \sum_{u \in \mathcal{A}_i } \swh{u} = 
\sum_{v \in s_i \mathcal A_i} \swh{v} - \sum_{u \in \mathcal A_i} \swh{u}
\]
  is nonzero. Therefore, we get  $s_i \cdot \swh{s_i} \neq  \swh{s_i}$.
\end{proof}

\begin{proposition}\label{proposition_Ai}
For $i \in [n-1]$, we have
\[
\mathcal{A}_i = \{ u \in \mathfrak{S}_n \mid u^{-1}(i+1) \leq i, ~~ u^{-1}(i) >
i, ~~ h(u^{-1}(i+1)) < u^{-1}(i), ~~\ell_h(u) = 1\}.
\]
\end{proposition}
\begin{proof}
	Recall from Proposition~\ref{prop:simple action} that we have
	\begin{equation}\label{eq_Ai_1}
	\mathcal{A}_i = \{ u \in \Owh{e}^T \cap \Omega_{s_i}^T \mid
	\dim_{\C} \mathcal{T}_u = \dim_{\C} \Owh{s_i}, ~~u \dasharrow s_i u \}.
	\end{equation}
	By~\eqref{eq_Tu},  we get $\textrm{codim}_{\C} \mathcal{T}_u = \textrm{codim}_{\C}
	\Owh{u}$. Therefore, $\textrm{codim}_{\C} \mathcal T_u = \textrm{codim}_{\C} \Ow{s_i} =1$ if and only if $\ell_h(u)= 1$.
	Moreover,  because $\Owh{e}^T = \mathfrak{S}_n$, we get
	\[
	\Owh{e}^T \cap \Omega_{s_i}^T = \mathfrak{S}_n \cap \Omega_{s_i}^T = \{ u
	\mid u \geq s_i\}.
	\]
	Now the description of $\mathcal{A}_i$ in~\eqref{eq_Ai_1} becomes
	\begin{equation}\label{eq_Ai_2}
	\mathcal{A}_i = \{ u \in \mathfrak{S}_n \mid u \geq s_i,  ~~\ell_h(u) = 1,
	~~u \dasharrow s_i u\}.
	\end{equation}
	
	Recall from~\cite[\S 3.2]{BilleyLakshmibai} a property of the Bruhat order.
	For a set $\{a_1,\dots,a_k\}$ of distinct integers,
	$\{a_1,\dots,a_k\}\!\!\uparrow$ denotes the ordered $k$-tuple obtained from
	$\{a_1,\dots,a_k\}$ by arranging its elements in ascending order. Moreover,
	we use the order to compare two ordered $k$-tuples defined as follows:
	$(a_1,\dots,a_k) \geq (b_1,\dots,b_k)$ if and only if $a_i \geq b_i$ for
	all $1 \leq i \leq k$.
	For a permutation $w \in \mathfrak{S}_n$, we denote by $w^{(k)}$ for the
	ordered $k$-tuple $\{w(1),\dots,w(k)\}\!\!\uparrow$.
	For $w_1, w_2 \in \mathfrak{S}_n$, we have $w_1 \geq w_2$ in Bruhat order
	if and only if
	\begin{equation}\label{eq_Bruhat_order}
	w_1^{(k)} \geq w_2^{(k)} \quad \text{ for all }1 \leq k \leq n.
	\end{equation}
		
	A permutation $u$ satisfies $u \geq s_i$  if and only if $u^{(k)} \geq
	(s_i)^{(k)} $ for all $k$. Note that we have
	\[
	(s_i)^{(k)} = \begin{cases}
	(1,2,\dots,k) & \text{ if } k \neq i,\\
	(1,2,\dots,i-1,i+1) & \text{ if } k =i.
	\end{cases}
	\]
	Since $(1,2,\dots,k)$ is the minimum among the ordered $k$-tuples of
	integers, for a permutation~$u$,
\begin{equation}\label{eq_u_geq_si_condition}
u \geq s_i
\iff u^{(i)} \geq (1,2,\dots,i-1,i+1).
\end{equation}

We set
\[
\mathcal B = \{ u \in \mathfrak{S}_n \mid u^{-1}(i+1) \leq i, ~~ u^{-1}(i) >
i, ~~ h(u^{-1}(i+1)) < u^{-1}(i), ~~\ell_h(u) = 1\}
\]
in the statement of the proposition. We first claim that $\mathcal B \subset \mathcal A_i$. Take $u \in \mathcal B$. Since $u^{-1}(i+1) \leq i$, the number $i+1$ appear in the first $i$ letters so we have $u^{(i)} \geq (1,2,\dots,i-1,i+1)$. This implies that $u \geq s_i$ by \eqref{eq_u_geq_si_condition}. Moreover, the condition $u \dasharrow s_iu$ is equivalent to saying that in one-line
	notation of $u$, the number $i+1$ appears ahead of $i$ and the locations of
	$i$ and $i+1$ in the one-line notation of $u$ are far from each other. More
	precisely, $u^{-1}(i+1) < u^{-1}(i)$ and $h(u^{-1}(i+1)) < u^{-1}(i)$.
Accordingly, $u$ is an element of $\mathcal A_i$.

We claim that $\mathcal A_i \subset \mathcal B$. Take $u \in \mathcal A_i$. Because of the condition $u \dasharrow s_iu$, we have  $u^{-1}(i+1) < u^{-1}(i)$ and $h(u^{-1}(i+1)) < u^{-1}(i)$. It is enough to check that $u^{-1}(i+1) \leq i$ and $u^{-1}(i) > i$.
Since $h$ satisfies $h(j) \geq j+1$ for all $j$, we have
\[
\ell_h(u) \geq \des(u).
\]
Accordingly, because of the assumption $\ell_h(u) = 1$, there exists only one descent in $u$.
Now assume on the contrary that $u^{-1}(i+1) > i$.  Then there exists $k > i+1$ such that $u^{-1}(k) \leq i$ since $u \geq s_i$ and~\eqref{eq_u_geq_si_condition}.
This implies that the numbers $k > i+1 > i$ satisfy $u^{-1}(k) < u^{-1}(i+1) < u^{-1}(i)$. This produces at least two descents in $u$. Therefore, we have $u^{-1}(i+1) \leq i$.
Now consider the condition $u^{-1}(i) > i$. Assume on the contrary that $u^{-1}(i) \leq i$ so we get $u^{-1}(i+1)<u^{-1}(i) \leq i$. Therefore, there exists $k < i$ such that $u^{-1}(k) > i$. This implies that the numbers $i+1 > i > k$ satisfy $u^{-1}(i+1) < u^{-1}(i) < u^{-1}(k)$. This produces at least two descents in $u$. Therefore, we have $u^{-1}(i) > i$.
This proves $\mathcal A_i \subset \mathcal B$ so we are done.
\end{proof}

\begin{example}\label{example_235666_Ai}
Let $h = (2,3,5,6,6,6)$. Using the description of $\mathcal{A}_i$ in
Proposition~\ref{proposition_Ai}, we have the following computations.
\begin{enumerate}
\item $\mathcal{A}_1 = \{23\vred1456, 24\vred1356, 25\vred1346,
26\vred1345, 234\vred156\}$,
\item $\mathcal{A}_2 = \{3\vred12456, 34\vred1256, 35\vred1246,
36\vred1245, 134\vred256\}$,
\item $\mathcal{A}_3 = \{4\vred12356, 14\vred2356, 24\vred1356, 45\vred1236,
46\vred1235\}$,
\item $\mathcal{A}_4 = \{5\vred12346, 15\vred2346, 25\vred1346, 35\vred1246,
56\vred1234\}$,
\item $\mathcal{A}_5 = \{6\vred12345, 16\vred2345, 26\vred1345, 36\vred1245, 46\vred1235\}$.
\end{enumerate}
We decorate the places where descents appear.
\end{example}

Proposition~\ref{proposition_Ai} provides the following corollary.
\begin{corollary}\label{cor_si_action_on_sigma_si}
For $i \in [n-1]$,
let $\mathcal A_i$ be the set used in the description~\eqref{eq_si_action_modified}.
Then, for $u \in \mathcal A_i$, we have $\swh{s_i u} = s_i \cdot \swh{u}$. In particular, we get
\[
s_i \cdot \swh{s_i} = \swh{s_i} + \sum_{u \in \mathcal A_i}\swh{u} -
\sum_{u \in \mathcal A_i} s_i \cdot \swh{u}. 
\]
\end{corollary}
\begin{proof}
As we have seen in~\eqref{eq_Ai_2} in the proof of Proposition~\ref{proposition_Ai}, $u \dasharrow s_i u$ for any $u \in \mathcal A_i$. 
Accordingly, we obtain
\[
s_i \cdot \swh{u} = \swh{s_iu}
\]
by Proposition~\ref{prop:simple action}(1). Using~\eqref{eq_si_action_modified}, the result follows.
\end{proof}

\begin{lemma}\label{lem:empty A_i}
If $T = [n-1]$, that is, $h(i)> i+1$ for all $i=1, 2, \dots, n-2$, then $\mathcal{A}_k=\varnothing$ for all $k=1, 2, \dots, n-1$.
\end{lemma}
\begin{proof}
Assume on the contrary that there exists $w\in \mathcal{A}_k$. Let $i_{k+1}=w^{-1}(k+1)$ and $i_k=w^{-1}(k)$.  Then, by Proposition~\ref{proposition_Ai} and the assumption,
\[
i_{k+1}\leq k <i_{k} \quad \text{ and } \quad i_{k+1}+1<h(i_{k+1})<i_k
\]
must hold and we can see that $i_{k+1}+2\leq i_k$ and the unique descent, say $d$, must be between $i_{k+1}$ and $i_k$, that is, $i_{k+1}\leq d<i_k$.

If $d=i_{k+1}$, then we have
\[
w(d) =k+1 >w(d+1) \quad \text{ and }\quad
 w(d+2)\leq w(i_k)=k.
\]
Hence, both $(d,d+1)$ and $(d, d+2)$ are counted in $\ell_h(w)$ contradicting to $w\in \mathcal{A}_k$.

If $d>i_{k+1}$, then we obtain
\[
w(i_{k+1})=k+1 \le w(d-1) <w(d) \quad
\text{ and } \quad
w(d+1)\leq w(i_k) = k.
\]
Hence, both $(d, d+1)$ and $(d-1, d+1)$ are counted in $\ell_h(w)$  contradicting to $w\in \mathcal{A}_k$.

This shows that there is no $w$ in $\mathcal{A}_k$.
\end{proof}

With explicit descriptions of $\mathcal{A}_i$ for $i=1, 2, \dots, n-1$, we completely determine the stabilizer subgroup of $\swh{\wi{i}}$ in most cases: See Propositions~\ref{prop:stabilizer_12}, and \ref{prop:stabilizer_34}.
The following lemma can be derived using known results given at the beginning of Section~8.5 and Theorem~3.3A in~\cite{DM}. We provide a proof for the readers' convenience.
\begin{lemma}\label{lem:subgroup of S_n}
Let $H$ be a subgroup of $\mathfrak S_n$, $n>2$,  that contains $\mathfrak S_{(i, n-i)}$ but does not contain the transposition $s_i$ for some $i$.

\begin{enumerate}
\item If $i\neq  n-i$, then $H=\mathfrak S_{(i, n-i)}$.
\item If $i = n-i$ and there is no $\alpha \in H$ such that $\alpha \left(\{1, 2, \dots, i\}\right)=\{i+1, i+2, \dots, n\}$, then $H=\mathfrak S_{(i, n-i)}$.
\end{enumerate}
\end{lemma}
\begin{proof}
Assume on the contrary that there is $\beta\in H \setminus \mathfrak S_{(i, n-i)}$.
If $i\neq  n-i$, we may assume that $i>n-i$ and  there must be elements $k, l\in\{1, \dots, i\}$
such that
\begin{equation}\label{eq_beta_kl}
\beta(k)\in\{1, \dots, i\}\quad \text{ and } \quad
\beta(l)\not\in\{1, \dots, i\}.
\end{equation}
On the other hand, if $i = n-i$ and there is no $\alpha \in H$ such that $\alpha \left(\{1, 2, \dots, i\}\right)=\{i+1, i+2, \dots, n\}$, then there also exist elements $k, l\in\{1, \dots, i\}$ satisfying~\eqref{eq_beta_kl}.

The element $\beta$ produces
$\beta (k\,\, l) \beta^{-1}=(\beta(k) \,\, \beta(l))\in H$ and we have
\[
(\beta(l)\,\, i+1)(\beta(k) \,\, i)(\beta(k) \,\, \beta(l))(\beta(k) \,\, i)(\beta(l)\,\, i+1)=(i\,\, i+1)=s_i\in H,
\]
which is a contradiction.
Here, we denote by $(k \ l)$ the transposition $s_{k,l}$.
We conclude that $H=\mathfrak S_{(i, n-i)}$.
\end{proof}

Let $\Stab(\swh{w})$ denote the stabilizer subgroup of $\swh{w}$ in  $\mathfrak S_n$.  We use Proposition~\ref{prop:simple action} to compute $\Stab(\swh{\wi{i}})$.

\begin{proposition}\label{prop:stabilizer_12}
Let $h \colon [n] \to [n]$ be a Hessenberg function and $i \in T = T_h$.
\begin{enumerate}
\item If $i+1 \in T \cup \{n\}$,
then
\begin{enumerate}
\item $\Stab(\swh{\wi{i}})=\mathfrak S_{(i,n-i)}$ if $n-i\neq i$ and $T \neq [n-1]$.
\item $\mathfrak S_{(i,n-i)}\leq \Stab(\swh{\wi{i}})<\mathfrak S_{n}$ if $n-i=i$ and $T \neq [n-1]$.
\item $\Stab(\swh{\wi{i}})=\mathfrak S_n$  otherwise, that is, $T = [n-1]$.
\end{enumerate}
\item If $i+1 \notin T \cup \{n\}$, then
$\Stab(\swh{\wi{i}})=\mathfrak S_{(n-1, 1)}$.
\end{enumerate}
\end{proposition}

\begin{proof}
\noindent (1)
Suppose that $i \in T$ and $i+1 \in T \cup \{n\}$, that is,  $h(i-1)> i$ and $h(i)>i+1$; or $i+1 = n$, then $\wi{i}=s_i$ by Definition~\ref{def:permutations}. For $k\neq i$, $s_k\wi{i}\rightarrow \wi{i}$ and $s_k \cdot \swh{\wi{i}}=\swh{\wi{i}}$. Accordingly,  we have 
\begin{equation}\label{eq_Stab_i_and_i+1_in_T}
\mathfrak{S}_{(i,n-i)} \leq \Stab(\swh{\wi{i}}) < \mathfrak{S}_n.
\end{equation}
On the other hand, we have $\wi{i}=s_i\rightarrow s_i\wi{i}=e$, so we apply Lemma \ref{lem: empty implies stabilize} to get that
\[
\mathcal A_i \text{ is empty  if and only if } s_i \in \Stab(\swh{\wi{i}}),
\]
where $\mathcal A_i$ is described in Proposition~\ref{proposition_Ai}.

Suppose that $T \neq [n-1]$. Choose $j \in [n-1] \setminus T$. Then $h(j-1) = j$.
If such $j$ satisfies $i < j$, then
\[
u=1\, 2\,\cdots i-1\,\, i+1\,\,\cdots \, j+1 \,\vred \, i\,\,j+2\,\cdots \, n
\]
is a permutation in $\mathcal A_i$ which has the unique descent at $j$.
If any $j \in [n-1] \setminus T$ satisfies $j < i$, then take
\[
u=i-j+1\,\,i-j+2\,\,\cdots\,\,i-1\,\,i+1\, \vred \, 1\,\,2\,\,\cdots i-j \,\,i\,\,i+2\,\,\cdots n.
\]
We notice that $u$ is a permutation in $\mathcal A_i$ which has the unique descent at $j$.
We thus have $\mathcal A_i\neq\varnothing$ and $s_i\not \in \Stab(\swh{\wi{i}})$ if $T \neq [n-1]$.
Lemma~\ref{lem:subgroup of S_n} and~\eqref{eq_Stab_i_and_i+1_in_T} now imply that $\Stab(\swh{\wi{i}})=\mathfrak S_{(i,n-i)}$ if $n-i\neq i$, which proves 1-(a). The statement 1-(b) follows from~\eqref{eq_Stab_i_and_i+1_in_T}.

If $T = [n-1]$,
then by Lemma~\ref{lem:empty A_i}, $\mathcal A_i$ is empty and $s_i\in \Stab(\swh{\wi{i}})$. Therefore, $\Stab(\swh{\wi{i}})=\mathfrak S_n$, which proves 1-(c).

\smallskip
\noindent (2)
Suppose that $i \in T$ and $i+1 \notin T \cup \{n\}$, that is,
$h(i)=i+1$. Then $s_k\wi{i}\rightarrow \wi{i}$ and $s_k\cdot \swh{\wi{i}}=\swh{\wi{i}}$ for $k=1, \dots, n-2$. Moreover, $\wi{i}\dashrightarrow s_{n-1}\wi{i}$ and we have $s_{n-1}\cdot \swh{\wi{i}}= \swh{s_{n-1}\wi{i}}\neq\swh{\wi{i}}$. This, due to Lemma~\ref{lem:subgroup of S_n},   proves the claim.
\end{proof}

We also obtain the following proposition using similar arguments used in the proof of Proposition~\ref{prop:stabilizer_12}.

\begin{proposition}\label{prop:stabilizer_34}
Let $h \colon [n] \to [n]$ be a Hessenberg function and $i \notin T = T_h$.
\begin{enumerate}
\item If $i+1 \in T \cup \{n\}$, then
$\Stab(\swh{\wi{i}})=\mathfrak S_{(1, n-1)} $.
\item If $i+1 \notin T \cup \{n\}$, then
\begin{enumerate}
\item $\Stab(\swh{\wi{i}})=\mathfrak S_{(n-i, i)} $ if $n-i\neq i$.
\item $\mathfrak S_{(n-i, i)}\leq \Stab(\swh{\wi{i}})<\mathfrak S_{n}$ if $n-i=i$.
\end{enumerate}
\end{enumerate}
\end{proposition}
\begin{proof}
\noindent
(1)
Suppose that $i \notin T$ and $i+1 \in T \cup \{n\}$.
Then we have $\wi{i}\dashrightarrow s_1\wi{i}$ and $s_k\wi{i}\rightarrow \wi{i}$ for $k=2, \dots, n-1$. Thus we have  $s_1\not\in \Stab(\swh{\wi{i}})$ and $s_2, \dots, s_{n-1}\in \Stab(\swh{\wi{i}})$, proving the first part of the proposition due to Lemma~\ref{lem:subgroup of S_n}.

\smallskip
\noindent
(2)
Suppose that $i \notin T$ and $i+1 \notin T \cup \{n\}$.
Then $\wi{i}\dashrightarrow s_{n-i}\wi{i}$ and $s_k\wi{i}\rightarrow \wi{i}$ for $k\neq n-i$. Thus we have $s_{n-i}\not\in \Stab(\swh{\wi{i}})$ and $s_k\in \Stab(\swh{\wi{i}})$ for $k\neq n-i$, proving the second part of the proposition due to Lemma~\ref{lem:subgroup of S_n}.
\end{proof}
\begin{remark}
We will see in Corollary~\ref{cor_stab_nn} that for $i$ satisfying $i \notin T$ and $i+1 \notin T \cup \{n\}$, we have $\Stab(\swh{\wi{i}}) = \mathfrak{S}_{(n-i,i)}$ while we have an inclusion in  Proposition~\ref{prop:stabilizer_34}.
\end{remark}

\section{A partition of \texorpdfstring{$\mathfrak{S}_n$}{Sn}}
\label{sec_partition}
In this section, we provide a partition of $\mathfrak{S}_n$ each class of which consists of permutations having the same graph type in Proposition~\ref{prop_elements_in_Pi}. Using this description, we also enumerate the dimension of $H^{2}(\Hess(S,h))$.
\begin{lemma}\label{lemma_Du_property}
	Let $T = T_h$.
	For $1 \leq i \leq n-1$, we have
\[
\left( \bigcup_{u \in \mathcal A_i} \Des(u) \right)\cap \{ j \in T \mid j
\geq i\} = \emptyset.
\]
\end{lemma}
\begin{proof}
Assume on the contrary that there exists $u \in \mathcal{A}_i$ such that the
element in the descent set $\Des(u) = \{j \}$ satisfies $j \in T$ and $j \geq
i$.
By the definition of $T$, we have $h(j-1) > j$, so a part of the graph
$\Gwh{u}$ can be depicted as follows:
\begin{center}
\begin{tikzpicture}[scale=1.5]
\tikzstyle{v} = [ ]

\node[v] (0) at (0,0) {{$1$}};
\node (1) at (1,0) {$\cdots$};
\node[v]  (2) at (2,0) {{$i$}};
\node (3) at (3,0) {$\cdots$};
\node[v] (4) at (4,0) {{$j-1$}};
\node[v] (5) at (5,0) {{$j$}};
\node[v] (6) at (6,0) {{$j+1$}};
\node (7) at (7,0) {$\cdots$};
\node[v] (8) at (8,0) {{ $n$}};

\draw[->,thick] (0)--(1);
\draw[->,thick] (1)--(2);
\draw[->,thick] (2)--(3);
\draw[->,thick] (3)--(4);
\draw[->,thick] (4)--(5);
\draw[->, out=45, in=135,thick] (4) to (6);
\draw[->,thick] (6)--(7);
\draw[->,thick] (7)--(8);

\draw[ <-, red,thick] (5)--(6);

\end{tikzpicture}
\end{center}
Indeed, the permutation $u$ satisfies
\begin{equation}\label{eq_u_description}
\begin{tikzcd}[column sep=0.1cm, row sep = 0.2cm]
u(1) \arrow[r, white,   "\textcolor{black}{<}" description]
&  \cdots   \arrow[r, white, "\textcolor{black}{<}" description]
&u(i)  \arrow[r, white, "\textcolor{black}{<}" description]
&\cdots \arrow[r, white, "\textcolor{black}{<}" description]
&u(j-1)  \arrow[rd, white,sloped, "\textcolor{black}{<}" description]
& <
&u(j+1)  \arrow[r, white, "\textcolor{black}{<}" description]
\arrow[ld, white,sloped, "\textcolor{black}{>}" description]
& u(j+2)  \arrow[r, white, "\textcolor{black}{<}" description]
& \cdots  \arrow[r, white, "\textcolor{black}{<}" description]
& u(n) \\
& & & & &u(j)
\end{tikzcd}
\end{equation}
Accordingly, $u(j) \geq j$ and
in the one-line notation of $u$, the numbers $[n] \setminus \{ u(j)\}$ are
displayed in ascending order. Therefore, $u(k) = k$ for all $1 \leq k < j$,
which contradicts the assumption $u \in \mathcal{A}_i$ because any elements
$u$ in $\mathcal{A}_i$ should satisfy $u^{-1}(i+1) \leq i$ and $h(u^{-1}(i+1))< u^{-1}(i)$ by
Proposition~\ref{proposition_Ai}.
Here, if $j = i$ and $u(j) = u(i) = i+1$, then we should have $u(i+1) = i$ and such a permutation $u$ cannot satisfy the condition $h(u^{-1}(i+1))< u^{-1}(i)$.
Hence, the result follows.
\end{proof}

\begin{example}\label{example_235666_Ai_intersection_T}
We demonstrate Lemma~\ref{lemma_Du_property} in this example.
Suppose that $h = (2,3,5,6,6,6)$.
We obtain
$T = T_h = \{ 1, 4,5\}$.
Considering the computations of $\mathcal{A}_i$ in Example~\ref{example_235666_Ai},
we obtain
\[
\begin{split}
&\left(\bigcup_{u \in \mathcal A_1} \Des(u)
\right) \cap \{ j \in T \mid j \ge 1\}
 = \{ 2,3\} \cap \{ 1,4,5\} = \emptyset,\\
&\left(\bigcup_{u \in \mathcal A_2} \Des(u)
\right) \cap \{ j \in T \mid j \ge 2\}
= \{1,2,3\} \cap \{4,5\} = \emptyset,\\
&\left(\bigcup_{u \in \mathcal A_3} \Des(u)
\right) \cap \{ j \in T \mid j \ge 3\}
= \{1,2\} \cap \{4,5\} = \emptyset, \\
&\left(\bigcup_{u \in \mathcal A_4} \Des(u)
\right) \cap \{ j \in T \mid j \ge 4\}
= \{1,2\} \cap \{4,5\} = \emptyset, \\
&\left(\bigcup_{u \in \mathcal A_5} \Des(u)
\right) \cap \{ j \in T \mid j \ge 5\}
= \{1,2\} \cap \{5\} = \emptyset.
\end{split}
\]
\end{example}

In the remaining part of this section, we provide a
\textit{geometric} construction of the permutation module decomposition of
$H^2(\Hess(S,h))$ exhibited in Theorem~\ref{thm:Chow}.
We separate the set~$[n-1]$ of indices into
the following two families:
$i \in T$ or $i \notin T$.

For $1 \leq i \leq n-1$, define a subset $P_i$ of $\mathfrak{S}_n$ by
\begin{equation}\label{eq_def_Pi}
P_i \colonequals  P_{\wi{i}}= \{ u \in \mathfrak{S}_n \mid G_{u,h} = G_{\wi{i},h}\}.
\end{equation}
By Propositions~\ref{prop:type and orbit} and \ref{prop:wi_elements_in_Gh1}, we have
$\{u \in \mathfrak{S}_n \mid \ell_h(u) = 1\} = P_1 \sqcup \cdots \sqcup
P_{n-1}$, and moreover,
\begin{equation}\label{eq_description_of_Pi}
P_i = \{ u \in \mathfrak{S}_n \mid \ell_h(u) = 1\} \cap \{ u \in \mathfrak{S}_n \mid \Des(u) = \{i\}\}.
\end{equation}
\begin{proposition}\label{prop_elements_in_Pi}
	For $1 \leq i \leq n-1$, we describe $P_i$ as follows.
	\begin{enumerate}
	\item If $i \in T$, then
	\[
	P_i = \begin{cases}
		\{\wi{i}\} & \text{ if } i+1 \in T \cup \{n\}, \\
		\{  s_{j+1} s_{j+2} \cdots s_{n-1} \wi{i} \mid i \leq j < n\}  & \text{ if
			} i+1 \notin T \cup \{n\}.
		\end{cases}
	\]
	\item If $i \notin T$, then
\[
P_i = \begin{cases}
\{ s_j s_{j-1} \cdots s_{1} \wi{i} \mid 0 \leq  j < i\}
& \text{ if }i+1 \in T \cup \{n\}, \\
\{ u \in \mathfrak{S}_n \mid \Des(u) = \{i\}\}
& \text{ if }i+1 \notin T \cup \{n\}.
\end{cases}
\]
	\end{enumerate}
\end{proposition}
\begin{proof}
Before providing case-by-case analysis, we recall from
Proposition~\ref{prop:type and orbit pre} that for any element $u \in P_i$, there
exists a sequence of simple reflections $s_{i_1},\dots,s_{i_r}$ such that $\wi{i}
\stackrel{s_{i_1}}{\dasharrow} \cdots \stackrel{s_{i_r}}{\dasharrow} u$.
Accordingly, by Lemma~\ref{lemma_graph_type}, to obtain all elements in $P_i$, it is enough to
consider simple reflections $s_j$ satisfying that $u \stackrel{s_j}{\dasharrow}
s_ju$ repeatedly for $u \in P_i$.
We notice that we may assume that $n > 2$ because when $n = 2$,
there is only one Hessenberg function satisfying $h(i) \geq i+1$, which is $h =
(2,2)$. In this case, the Hessenberg variety is the full flag variety $\flag(\C^2)$ and we have $T = \{1\}$ and $P_1 = \{ w^{[1]}\} = \{21\} $. This proves the statement when $n = 2$.

We provide a proof when $n>2$ using case-by-case analysis and recall the descriptions of~ $\wi{i}$ from
Definition~\ref{def:permutations} for each case.

\smallskip
\noindent \textbf{\textsf{Case 1-1.}} Suppose that $i \in T$ and $i+1 \in T
\cup \{n\}$.
Then, we have
\[
\wi{i} = 1\, 2\,\cdots i-1 \,\,i+1\, \vred \, i \,\, i+2\,\cdots\,\, n \in
\Ghone.
\]
Indeed, $\wi{i} = s_i$. In this case, there is no simple reflection $s_j$
satisfying
$s_i \stackrel{s_j}{\dasharrow} s_j s_i$. Therefore, $P_i = \{
\wi{i}\} = \{s_i\}$, proving the claim.

\smallskip
\noindent \textbf{\textsf{Case 1-2.}} Suppose that $i \in T$ and $i+1 \notin
T \cup \{n\}$.
Then we have
\[
\wi{i} = 1\,2\,\cdots i-1\,\, n \,\vred \, i\,\, i+1\,\cdots \, n-1 \in \Ghone.
\]
Since the stabilizer subgroup of $\swh{\wi{i}}$ is
$\mathfrak{S}_{(n-1,1)}$ by Proposition~\ref{prop:stabilizer_12}, it is enough to consider
the action of $s_{n-1}$ on $\wi{i}$.
\[
s_{n-1} \wi{i} = 1\,2\,\cdots i-1\,\, n-1 \,\vred \, i\,\cdots \, n-2 \,\, n.
\]
Since $\wi{i} \dasharrow s_{n-1} \wi{i}$, we obtain $s_{n-1} \wi{i} \in P_i$ by Lemma~\ref{lemma_graph_type}.

On the other hand, the class $\swh{s_{n-1}\wi{i}} = s_{n-1} \cdot \swh{\wi{i}}$ has the stabilizer $s_{n-1} \mathfrak{S}_{(n-1,1)} s_{n-1}=
\mathfrak{S}_{\{1,2,\dots,n-2,n\}} \times \mathfrak{S}_{\{n-1\}}$.
Therefore, $s_{n-2}$ is the only element we have to care about and
it produces an element $s_{n-2}s_{n-1} \wi{i} \in P_i$ because $\ell_h(s_{n-2}s_{n-1} \wi{i}) = 1$ and $\Des(s_{n-2}s_{n-1} \wi{i}) = \{ i \}$. Continuing a similar
process, we obtain
\[
s_{j} s_{j+1} \cdots s_{n-1} \wi{i} \in P_i \quad \text{ for }i < j \leq n.
\]
For the case when $j = i+1$, we obtain
\[
s_{i+1} \cdots s_{n-1} \wi{i} = 1\,2\,\cdots i-1\,\, i+1 \,\vred \, i\,\cdots
\, n-1 \,\, n = s_i.
\]
Since there is no simple reflection $s_j$ satisfying $s_i
\stackrel{s_j}{\dasharrow} s_j s_i$, we obtain
\[
P_i = \{  s_{j+1} \cdots s_{n-1} \wi{i} \mid i \leq j < n\},
\]
which proves the claim.

\smallskip
\noindent \textbf{\textsf{Case 2-1.}} Suppose that $i \notin T$ and $i+1 \in T
\cup \{n\}$.
Then we have
\[ \wi{i} =2\,\cdots \,i \,\,i+1\,\vred\, 1\,\,i+2 \, \cdots \, n \,\,  \in \Ghone.
\]	
In this case, the stabilizer subgroup of $\swh{\wi{i}}$
is $\mathfrak{S}_{(1,n-1)}$ by
Proposition~\ref{prop:stabilizer_34}. Hence, it is enough to consider the action
of $s_1$:
\[
s_1 \wi{i} = 1\,\,3\,\cdots \,i \,\,i+1\,\vred\, 2\,\,i+2 \, \cdots \, n  \,\,\in P_i.
\]
Continuing a similar process as in \textsf{Case~1-2}, we obtain
\[
s_j s_{j-1} \cdots s_1 \wi{i} \in P_i \quad \text{ for }0 \leq j < i.
\]
For the case when $j = i-1$, we obtain
\[
s_{i-1} s_{i-2} \cdots s_1 \wi{i} = 1\,\,2\,\cdots \,i-1 \,\,i+1\,\vred\,
i\,\,i+2 \, \cdots \, n  = s_i.
\]
Since there is no simple reflection $s_j$ satisfying $s_i
\stackrel{s_j}{\dasharrow} s_j s_i$, we obtain
\[
P_i = \{ s_j s_{j-1} \cdots s_1 \wi{i}\mid 0 \leq  j < i\},
\]
proving the claim.

\smallskip
\noindent \textbf{\textsf{Case 2-2.}} Suppose that $i \notin T$ and $i+1 \notin
T\cup \{n\}$.
Because of~\eqref{eq_description_of_Pi}, it is
enough to show that any element having one descent at $i$ satisfies $\ell_h(u)=1$.
Since the Hessenberg function $h$ is weakly increasing and $h(i-1)=i$, we have $h(k) \leq i$
for $k \leq i-1$. Accordingly, any $u \in \mathfrak{S}_n$ such that $\Des(u) = \{i\}$ satisfies $\ell_h(u) =1$, proving the claim.
\end{proof}

\begin{example}\label{example_235666_Pi}
Let $h = (2,3,5,6,6,6)$.
By Definition~\ref{def:permutations}, we have
\[
\wi{1} = 6 \vred 12345, \quad
\wi{2} = 56 \vred 1234, \quad
\wi{3} = 234 \vred 156, \quad
\wi{4} = 1235 \vred 46, \quad
\wi{5} = 12346 \vred 5.
\]
By Proposition~\ref{prop_elements_in_Pi}, we obtain
\[
\begin{split}
P_1 &= \{ s_{j+1} s_{j+2} \cdots s_5 \wi{1} \mid 1 \le j < 6\}
= \{ 6\vred 12345, 5\vred 12346, 4\vred 12356, 3\vred 12456, 2\vred 13456\}, \\
P_2 &= \{ u \in \mathfrak{S}_6 \mid \Des(u) = \{2\}\} \\
&= \{13\vred 2456, 14\vred 2356, 15\vred 2346, 16\vred 2345, 23\vred 1456, 24\vred 1356, 25\vred 1346, \\
& \qquad 26\vred1345, 34\vred1256, 35\vred1246, 36\vred1245,
45\vred1236, 46\vred1235, 56\vred1234 \}, \\
P_3 &= \{ s_j s_{j-1} \cdots s_1 \wi{3} \mid 0 \le j < 3 \} = \{
234\vred156, 134\vred256, 124\vred356\}, \\
P_4 &= \{ \wi{4} \} = \{ 1235\vred46\}, \\
P_5 &= \{ \wi{5} \} = \{ 12346\vred5\}.
\end{split}
\]
\end{example}

Proposition~\ref{prop_elements_in_Pi} leads us to compute the cardinality of the set $\{ u \in \mathfrak{S}_n \mid \ell_h(u)=1\}$, which is the same as the dimension of $H^2(\Hess(S,h))$.
For $i \in [n-1]$, we define $d_i$ to be
\begin{equation}\label{eq_def_di}
d_i = \begin{cases}
1 & \text{ if } i \in T, \\
n & \text{ if } i \notin T \text{ and } i+1 \in T \cup \{n\}, \\
{n \choose i} & \text{ if } i \notin T \text{ and } i+1 \notin T \cup \{n\}.
\end{cases}
\end{equation}
\begin{proposition}\label{prop_dimension}
The cardinality of the set $\{ u \in \mathfrak{S}_n \mid \ell_h(u)=1\}$ is $\sum_{i=1}^{n-1} d_i$, which is the same as the dimension of $H^2(\Hess(S,h))$.
\end{proposition}
\begin{proof}
By Proposition~\ref{prop_elements_in_Pi}, we know that the cardinality of each $P_i$ is given by
	\begin{equation}\label{eq_number_of_Pi}
	\lvert P_i \rvert = \begin{cases}
		1 & \text{ if } i \in T \text{ and } i+1 \in T \cup \{n\}, \\
		n-i & \text{ if } i \in T \text{ and } i+1 \notin T \cup \{n\}, \\
		i & \text{ if } i \notin T \text{ and } i+1 \in T \cup \{n\}, \\
		{n \choose i}-1 & \text{ if } i \notin T \text{ and } i+1 \notin T \cup \{n\}.
	\end{cases}
	\end{equation}
We consider indices $i \in T$.
We denote by $T = \{t_1 < t_2 < \dots < t_s\}$. Here, $s = \lvert T \rvert$.
Using the elements in $T$, we consider a partition
$(\mathcal{B}_1,\dots,\mathcal{B}_s)$ of the set $[n-1]$, where
\[
\mathcal{B}_a = [t_a, t_{a+1}) \quad \text{ for }1 \leq a \leq s.
\]
Here, we set $t_{s+1} = n$.
To prove the proposition, it is enough to prove that
\begin{equation}\label{eq_dim_claim}
\sum_{i \in \mathcal B_a} d_i = \sum_{i \in \mathcal B_a} \lvert P_i \rvert \quad \text{ for } 1 \le a \le s.
\end{equation}

We first consider the case when $\lvert \mathcal B_a \rvert = 1$, i.e., $\mathcal B_a = \{t_a\}$. In this case, $t_a \in T$ and $t_{a}+1 = t_{a+1}$ is again contained in $T \cup \{n\}$. Accordingly, $d_{t_a} = 1$ by~\eqref{eq_def_di} and $\lvert P_{t_a} \rvert = 1$ by~\eqref{eq_number_of_Pi}, providing~\eqref{eq_dim_claim}.

Now we suppose that $\lvert \mathcal B_a \rvert =x>1$, i.e., $\mathcal B_a = \{t_a, t_{a}+1,\dots, t_a + x-1\}$. In this case, $t_a \in T$ and $t_{a}+x = t_{a+1} \in T \cup \{n\}$, whereas $t_{a}+1,\dots,t_a+x-1 \notin T \cup \{n\}$. Considering $\vert P_i \rvert$ for $i \in \mathcal B_a$, by~\eqref{eq_number_of_Pi}, we obtain
\[
\lvert P_i \rvert
= \begin{cases}
	n-t_a & \text{ if } i = t_a, \\
	{n \choose i} -1 & \text{ if } t_a < i < t_a + x -1, \\
	t_a + x -1 & \text{ if } i = t_a + x -1.
\end{cases}
\]
This provides
\[
\begin{split}
\sum_{i \in \mathcal B_a} \lvert P_i \rvert &=
n-t_a +  \sum_{t_a < i < t_a+x-1} \left({n \choose i} - 1 \right)
+ t_a + x -1 \\
&= n + x -1 + \sum_{t_a < i < t_a + x -1} \left({n \choose i} -1 \right)\\
&= 1 +
\sum_{t_a < i < t_a + x -1}  {n \choose i}
+ n \\
&= \sum_{i \in \mathcal B_a} d_i,
\end{split}
\]
proving~\eqref{eq_dim_claim}. Hence, we are done.
\end{proof}

\section{Geometric construction}\label{sec_geometric_construction}
In this section, we consider a geometric construction of the permutation module representation of $H^2(\Hess(S,h))$ using the BB basis. Before going on, we recall the result by Chow~\cite{Chow_h2}.
Chow showed that $H^2(\Hess(S, h))$ is decomposed as a sum of permutation
modules by computing the $e$-expansion of the chromatic quasisymmetric function corresponding to the Hessenberg function $h$, which is equivalent to the following theorem.
\begin{theorem}[\cite{Chow_h2}]\label{thm:Chow}
 Let $h \colon [n] \to [n]$ be a Hessenberg function such that $h(i)\geq i+1$ for all $i<n$ and let $T = T_h$. Then,
 as $\mathfrak S_n$-modules,
\[
H^2(\Hess(S, h))\cong \bigoplus_{i=1}^{n-1} M^{\alpha^{i}}\,,
\]
where for $1\leq i\leq n-1$,
\[
 \alpha^i=\begin{cases}
	(n) & \text{ if } i \in T, \\
	(1, n-1) & \text { if } i \notin T \text{ and } i+1 \in T \cup \{n\}, \\
	(i, n-i) & \text { if } i \notin T \text{ and } i+1 \notin T \cup \{n\}.
		\end{cases}
\]
\end{theorem}

\begin{example}\label{example:23666788} 	
Let $n=8$ and $h=(2,3,6,6,6,7,8,8)$.
Then, $T = T_h = \{1,4,5\}$ and we obtain
\begin{gather*}
\wi{1} = 8\vred 1234567, \quad
\wi{2} = 78\vred 123456,\quad
\wi{3} = 234\vred 15678,\quad
\wi{4} = 1235\vred 4678, \\
\wi{5} = 12348\vred 567,\quad
\wi{6} = 345678\vred 12,\quad
\wi{7} = 2345678\vred 1.
\end{gather*}
Moreover, we have
\[
\alpha^1=(8),\quad \alpha^2=(2, 6), \quad\alpha^3=(1, 7), \quad\alpha^4=(8),\quad
\alpha^5=(8), \quad\alpha^6=(6, 2),\quad \alpha^7=(1, 7).
\]
We therefore have
\[
H^2(\Hess(S, h))\cong  M^{(8)}\oplus M^{(2, 6)} \oplus M^{(1, 7)}
\oplus  M^{(8)}  \oplus  M^{(8)}\oplus M^{(6, 2)}\oplus
M^{(1, 7)}\,
\]
by Theorem~\ref{thm:Chow}.
\end{example}

We will construct a basis $\{\hatswh{w} \mid w \in \Ghone\}$ of
$H^2(\Hess(S,h))$ generating permutation modules in the right hand side of the
formula in Theorem~\ref{thm:Chow} by modifying the basis element $\swh{\wi{i}}$.
We are going to modify $\swh{\wi{i}}$ in such a way that its stabilizer subgroup agrees
with the Young subgroup $\mathfrak{S}_{\alpha^i}$, where $\alpha^i$ is the
composition defined in Theorem~\ref{thm:Chow}.
\begin{definition}\label{def_hat_wi}
Let $h \colon [n] \to [n]$ be a Hessenberg function such that $h(i) \geq i +1$ for all $i < n$.
For $ 1 \leq i \leq n-1$, we define $\hatswh{i}$ as follows.
\[
\hatswh{i} \colonequals
\begin{cases}
\displaystyle \sum_{v \in \mathfrak{S}_n} v \cdot  \swh{\wi{i}} & \text{ if }i \in T,\\
\swh{\wi{i}} &\text{ otherwise}.
\end{cases}
\]
\end{definition}

Now we state the main theorem of this section which provides a \textit{geometric} construction of the
permutation module decomposition of $H^2(\Hess(S,h))$ exhibited in
Theorem~\ref{thm:Chow}.
\begin{theorem}\label{thm_H2}
	For $1 \leq i \leq n-1$, let $M_{i,h} = \C \mathfrak{S}_n(\hatswh{i})$ be the $\mathfrak{S}_n$-module generated by $\hatswh{i}$.
Then we have the following.
\begin{enumerate}
\item $M_{i,h} \cong M^{\alpha^i}$;
\item $\Stab(\hatswh{i}) = \mathfrak{S}_{\alpha^i}$; and
\item $\displaystyle H^2(\Hess(S,h)) =\bigoplus_{i=1}^{n-1} M_{i,h}$.
\end{enumerate}
\end{theorem}
Before providing a proof of Theorem~\ref{thm_H2}, we prepare terminologies and three lemmas. 
We consider the sum 
\[
M \colonequals \sum_{i=1}^{n-1} M_{i,h}
\] 
of the modules which is
contained in $H^2(\Hess(S,h))$  by the construction.
\begin{definition}
For two elements $\tau_1$ and $\tau_2$ in $H^2(\Hess(S,h))$, we say that they are the same \emph{modulo $M$}, denoted by 
\[
\tau_1 \equiv \tau_2 \mod M,
\] 
if their difference $\tau_1 - \tau_2$ is contained in $M$. \end{definition}
Since $M$ is invariant under the $\mathfrak{S}_n$-action, for any $u \in \mathfrak{S}_n$, if $\tau_1 \equiv \tau_2 \mod M$, then $u \cdot \tau_1 \equiv u \cdot \tau_2 \mod M$.
By Proposition~\ref{prop:type and orbit}, we obtain the following lemma.
\begin{lemma}\label{lemma_main_1}
For $i \notin T$ and $u \in P_i$, we have
\begin{equation*}
\swh{u} \in M_{i,h} \subset M.
\end{equation*}
\end{lemma}
\begin{proof}
As is defined in Definition~\ref{def_hat_wi}, the class $\hats_{i,h}$ is $\swh{\wi{i}}$. On the other hand, by Proposition~\ref{prop:type and orbit}, any class $\swh{u}$ for $u \in P_{\wi{i}} = P_i$ is contained in the $\mathfrak{S}_n$-module $M(\swh{i}) = \C \mathfrak{S}_n(\swh{\wi{i}})$. This proves the lemma.
\end{proof}

\begin{lemma}\label{lemma_main_2}
For $i \in [n-1]$, 
the class $s_i \cdot \swh{s_i}$ can be expressed as follows:
\begin{equation}\label{eq_si_swh_si}
s_i \cdot \swh{s_i}
= \swh{s_i} +
\underbrace{\sum_{\substack{u \in P_{\ell} \\
\ell \notin T}} a_u \swh{u} + \sum_{\substack{u \in P_{\ell} \\ \ell \in T,
\ell < i}} b_u \swh{u}}_{\text{$=: \tau_i$}} \quad \text{ for some } a_u, b_u \in \C.
\end{equation}
\end{lemma}
\begin{proof}
By~\eqref{eq_si_action_modified}, we have
\begin{equation}\label{eq_lemma_main_2_1}
s_i \cdot \swh{s_i}
= \swh{s_i} + \sum_{v \in \mathcal A_i} \swh{v} - \sum_{v \in \mathcal A_i} \swh{s_iv}.
\end{equation}
The classes determined by $\mathcal A_i$ and $s_i \mathcal A_i$ are all of degree $1$, that is,
\[
\mathcal A_i \sqcup s_i \mathcal A_i   \subset P_1 \sqcup P_2 \sqcup \cdots \sqcup P_{n-1}.
\]
Moreover, by Corollary~\ref{cor_si_action_on_sigma_si}, for any $v \in \mathcal{A}_i$, we have $\Des(v) = \Des(s_iv)$ and hence $v$ and $s_iv$ are in the same $P_{\ell}$.
Accordingly, the second and third sum in the equation~\eqref{eq_lemma_main_2_1} can be expressed as follows:
\[
\begin{split}
\sum_{v \in \mathcal A_i} \swh{v} - \sum_{v \in \mathcal A_i} \swh{s_iv}
&= \sum_{v \in \mathcal A_i} (\swh{v} - \swh{s_iv})\\
&= \sum_{\ell=1}^{n-1} \sum_{v \in \mathcal A_i \cap P_{\ell}}
(\swh{v} - \swh{s_iv}).
\end{split}
\]
By applying Lemma~\ref{lemma_Du_property} and~\eqref{eq_description_of_Pi}, we get $\mathcal A_i \cap P_{\ell} = \emptyset$ if $\ell \in T$ and $\ell \geq i$, so the result follows.
\end{proof}
We let $\tau_i$  be the sum of the second and third terms in~\eqref{eq_si_swh_si} for later use.
To provide the last lemma, we prepare terminologies. 
Suppose that $i \in T = \{t_1 < t_2 < \dots < t_s\}$ and $i+1 \notin T \cup \{n\}$.  
In this case, $P_i = \{  s_{j+1} s_{j+2}\cdots s_{n-1} \wi{i} \mid  i \leq j < n\}$ by
Proposition~\ref{prop_elements_in_Pi}(1).
Here, we notice that the one-line notation of an element in $P_i$ is given as follows:
\begin{equation}\label{eq_rho_j_per}
s_{j+1} s_{j+2} \cdots s_{n-1}  \wi{i} = 1\,2\,\cdots i-1\,\, j+1 \,\vred \, i\,\, i+1\,\cdots \, j \, j+2 \, \cdots \, n \quad \text{ for } i\leq j < n.
\end{equation}
We denote by
\begin{equation}\label{eq_def_rho_j}
\rho_j \colonequals  \swh{ s_{j+1} s_{j+2} \cdots s_{n-1} \wi{i}} \quad \text{ for } i \leq j < n.
\end{equation}
By the proof of Proposition~\ref{prop_elements_in_Pi}, the classes $\rho_j$ satisfy the relations
\begin{equation}\label{eq_rhoj_relation}
\begin{split}
&s_j \cdot \rho_j = \rho_{j-1} \quad \text{for $i < j < n$}, \\ 
& s_k \cdot  \rho_j = \rho_j \quad \text{ for }i \leq j < n \text{ and }k \neq j, j+1.
\end{split}
\end{equation}

\begin{lemma}\label{lemma_main_3}
Suppose that $i \in T$ and $i+1 \notin T \cup \{n\}$.
For $j$ satisfying $i = t_a < j < t_{a+1}$, 
we get
\begin{equation}\label{eq_rhoj_rhoj-1}
s_j \cdot \swh{s_j} 
= \swh{s_j} +\sum_{\substack{v \in P_{\ell} \\ \ell \notin T}} a'_v \swh{v} + (\rho_j - \rho_{j-1})
+ \sum_{\substack{v \in P_{\ell} \\
\ell \in T, \ell < i}} b'_v \swh{v} \quad \text{ for some }a_v', b_v' \in \C.
\end{equation}
\end{lemma}
\begin{proof}
Using the formula~\eqref{eq_si_swh_si} in Lemma~\ref{lemma_main_2}, we have
\begin{equation*}
\begin{split}
s_j \cdot \swh{s_j} &= \swh{s_j}
+\sum_{\substack{v \in P_{\ell} \\ \ell \notin T}} a'_v \swh{v} + \sum_{\substack{v \in P_{\ell} \\
\ell \in T, \ell < j}} b'_v \swh{v}\\
&= \swh{s_j}
+\sum_{\substack{v \in P_{\ell} \\ \ell \notin T}} a'_v \swh{v}
+ \sum_{v \in P_i
} b'_v \swh{v} + \sum_{\substack{v \in P_{\ell} \\
\ell \in T, \ell < i}} b'_v \swh{v}\\
&= \swh{s_j} +\sum_{\substack{v \in P_{\ell} \\ \ell \notin T}} a'_v \swh{v} + (\rho_j - \rho_{j-1})
+ \sum_{\substack{v \in P_{\ell} \\
\ell \in T, \ell < i}} b'_v \swh{v}\\
\end{split}
\end{equation*}
for some $a_v',b_v' \in \C$.
Here, the third equality comes from the following:  For $i = t_a < j < t_{a+1}$,
we have 
\begin{equation}\label{eq_Aj_cap_Pi}
\mathcal A_j \cap P_i = \{s_{j+1} s_{j+2} \cdots s_{n-1} \wi{i}\}
\end{equation}
because of Proposition~\ref{proposition_Ai},~\eqref{eq_rho_j_per}, and moreover, $s_j \cdot \rho_j = \rho_{j-1}$ (see~\eqref{eq_rhoj_relation}). This proves the lemma.
\end{proof}

\begin{proof}[Proof of Theorem~\ref{thm_H2}]
Let $M = \sum_{i=1}^{n-1} M_{i,h}$ be the sum of the modules which is
contained in $H^2(\Hess(S,h))$  by the construction.
We first notice that
for each $1 \le i \le n-1$, we have
\[
\dim_{\C} M_{i,h} \leq \frac{|\mathfrak{S}_n|}{|\Stab(\hats_{i,h})|} \le  \dim_{\C}
M^{\alpha^i}
\]
by the orbit-stabilizer theorem and Proposition~\ref{prop:stabilizer_34}.
Hence, we obtain
\[
\dim_{\C} M \le \sum_{i=1}^{n-1} \dim_{\C} M_{i,h}
\le \sum_{i=1}^{n-1} \dim_{\C} M^{\alpha^i} = \dim_{\C} H^2(\Hess(S,h)).
\]
Here, the last equality comes from Proposition~\ref{prop_dimension}.
Accordingly, if we prove $M = H^2(\Hess(S,h))$, then we get
\begin{enumerate}
\item $M_{i,h} \cong M^{\alpha^i}$;
\item $\Stab(\hats_{i,h}) = \mathfrak{S}_{\alpha^i}$; and\label{stab}
\item $M$ is a direct sum of $M_{i,h}$, that is, $M = \bigoplus_{i=1}^{n-1} M_{i,h} = H^{2}(\Hess(S,h))$,
\end{enumerate}
proving the theorem.

Note that
	\[
	H^2(\Hess(S,h)) = \span_{\C} \{ \swh{u} \mid \ell_h(u) = 1\} =
	\span_{\C}\{ \swh{u} \mid u \in P_1 \sqcup P_2 \sqcup \cdots \sqcup
	P_{n-1}\}.
	\]

\smallskip
\noindent \textsf{\underline{Claim}:} For any $i \in [n-1]$ and $u \in P_i$, we have $\swh{u} \in M$.
\smallskip

If the claim holds, then we get the
	desired result $H^2(\Hess(S,h)) \subset M$.
We provide a proof using case-by-case analysis.

\smallskip
\noindent \noindent \textsf{\textbf{Case 1.}}
	First of all, if $i \notin T$, then
\begin{equation*}
\swh{u} \in M_{i,h} \subset M \quad \text{ for }u \in P_i
\end{equation*}
by Lemma~\ref{lemma_main_1}.
This proves the claim for $i \notin T$.

We note that since $\swh{s_i} \in M_{i,h}$, we have
\begin{equation}\label{eq_si_swh_si_in_M}
s_i \cdot \swh{s_i} \in M_{i,h} \subset M \quad \text{ for }i\notin T.
\end{equation}

\smallskip
\noindent \noindent \textsf{\textbf{Case 2.}}
We now consider indices $i \in T$.
	We denote by $T = \{t_1 < t_2 < \dots < t_s\}$. Here, $s = \lvert T \rvert$. 
We will prove the claim using an induction argument on $1 \leq a \leq s$. 

We first consider the case when $a = 1$, that is, $ i = t_1 = 1$ and consider two cases separately: 
\[
2 \in T \cup \{n\};\text{ or }\quad 2 \notin T \cup \{n\}. 
\]

\smallskip
\noindent \textsf{\textbf{Case 2-1.}} Suppose that $i = 1 \in T$ and $i+1 \in T \cup \{n\}$.
In this case, $P_i = \{ \wi{i} \}$, and moreover, $\wi{i} = s_i$ by
Proposition~\ref{prop_elements_in_Pi}(1).
The class $\hats_{i,h}$ is defined by
\[
\hats_{i,h} = \sum_{v \in \mathfrak{S}_n} v \cdot \swh{s_i} \in M.
\]
By Proposition~\ref{prop:stabilizer_12}, the stabilizer subgroup of $\swh{s_i}$
is $\mathfrak{S}_n$; or it is a proper subgroup containing $\mathfrak{S}_{(i,n-i)}$. If the stabilizer
subgroup is $\mathfrak{S}_n$, then $\hats_{i,h} = n! \swh{s_i} \in M$, proving that the class $\swh{s_i}$, the only element in $P_i$, is
contained in $M$.

On the other hand, suppose that the stabilizer subgroup of $\swh{s_i}$ is not
the whole group~$\mathfrak{S}_n$ but contains $\mathfrak{S}_{(i,n-i)}$.
In this case, the third term on the right-hand side of~\eqref{eq_si_swh_si} does not exist. Moreover, because of Lemma~\ref{lemma_main_1}, the second term on the right-hand side of~\eqref{eq_si_swh_si} is contained in $M$, so $\tau_i \in M$. This proves
\begin{equation}\label{eq_case2-1_modulo}
s_i\cdot  \swh{s_i} \equiv \swh{s_i} \mod M.
\end{equation}
Moreover, because the stabilizer group of $\swh{s_i}$ contains $\mathfrak{S}_{(i,n-i)}$, we have
\begin{equation}\label{eq_case2-1_modulo2}
s_j \cdot \swh{s_i} = \swh{s_i} \quad \text{ for } j \neq i.
\end{equation}
Considering equations~\eqref{eq_case2-1_modulo} and~\eqref{eq_case2-1_modulo2}, for any $v \in \mathfrak{S}_n$, we have $v \cdot \swh{s_i} \equiv \swh{s_i} \mod M$. Accordingly, we obtain
\[
\hats_{i,h} =
\sum_{v \in \mathfrak{S}_n} v \cdot  \swh{s_i}
\equiv \sum_{v \in \mathfrak{S}_n} \swh{s_i} = n! \swh{s_i} \mod M.
\]
Since $\hats_{i,h} \in M$, the class $\swh{s_i}$, the only element in $P_i$, is
contained in $M$.
This proves the claim when $i=1 \in T$ and $i+1 \in T \cup \{n\}$.

\smallskip
\noindent \textsf{\textbf{Case 2-2.}} Suppose that $i = 1 \in T$ and $i+1 \notin T \cup \{n\}$.
By Proposition~\ref{prop:stabilizer_12}(2), we have
\[
\begin{split}
\hatswh{i} &= \sum_{v \in \mathfrak{S}_n} v \cdot \swh{\wi{i}} \\
&= (n-1)!(\swh{\wi{i}} + s_{n-1} \cdot \swh{\wi{i}} + (s_{n-2}s_{n-1})\cdot \swh{\wi{i}} + \cdots + (s_1 \cdots s_{n-1})\cdot \swh{\wi{i}}) \\
&= (n-1)!(\rho_{n-1}+ \rho_{n-2} + \cdots + \rho_{i+1}+\rho_i
+ s_i \cdot \rho_i + (s_{i-1}s_i) \cdot \rho_i  + \cdots + (s_1 \cdots s_i) \cdot \rho_i) \in M.
\end{split}
\]
Here, $\rho_j$'s are the classes defined in~\eqref{eq_def_rho_j}, and the last equality comes from~\eqref{eq_rhoj_relation}. Accordingly, since $\hatswh{i} \in M$, the following class $\hatswh{i}'$ is contained in~$M$.
\begin{equation}\label{eq_hatswhi_in_M}
\hatswh{i}' \colonequals \rho_{n-1} + \rho_{n-2} + \cdots + \rho_{i+1}+\rho_i
+ s_i \cdot \rho_i + (s_{i-1}s_i) \cdot \rho_i  + \cdots + (s_1 \cdots s_i) \cdot \rho_i \in M.
\end{equation}
Since $i = 1$, the third term on the right-hand side of~\eqref{eq_si_swh_si} does not exist. Moreover, because of Lemma~\ref{lemma_main_1}, the second term on the right-hand side of~\eqref{eq_si_swh_si} is contained in $M$, so $\tau_i \in M$. This provides the modular equality~\eqref{eq_case2-1_modulo}, that is, $s_i \cdot \swh{s_i} \equiv \swh{s_i} \mod M$, 
and
\begin{equation*}
\begin{split}
(s_k s_{k+1} \cdots s_{i-1} s_i) \cdot \rho_i
&= (s_k s_{k+1} \cdots s_{i-1} s_i) \cdot \swh{s_i} \\
&\equiv (s_k s_{k+1} \cdots s_{i-1})\cdot \swh{s_i} \mod M
\end{split}
\end{equation*}
for $1 \le k \le i-1$.
Accordingly, we obtain
\begin{equation}\label{eq_trivial_class_hatswhi}
\hatswh{i}'  \equiv  \rho_{n-1} +  \rho_{n-2} + \cdots +
 \rho_{i+1}+ (i+1) \rho_i \mod M,
\end{equation}
and moreover, by~\eqref{eq_hatswhi_in_M}, we get
\begin{equation}\label{eq_rho_sum_in_M}
\rho_{n-1} +  \rho_{n-2} + \cdots +
 \rho_{i+1}+ (i+1) \rho_i \in M.
\end{equation}

On the other hand, for $1=t_1 < j < t_2 -1$, by Lemma~\ref{lemma_main_3}, we get
\begin{equation}\label{eq_sj_sigma_sj_and_rho}
\begin{split}
s_j \cdot \swh{s_j} 
&= \swh{s_j} +\sum_{\substack{v \in P_{\ell} \\ \ell \notin T}} a'_v \swh{v} + (\rho_j - \rho_{j-1})
+ \sum_{\substack{v \in P_{\ell} \\
\ell \in T, \ell < i}} b'_v \swh{v} \quad \text{ for some }a_v', b_v' \in \C \\
&\equiv \rho_j - \rho_{j-1} \mod M.
\end{split}
\end{equation}
Here, the modular equality holds because the fourth term does not exist (since $i = 1$) and the second term and $\swh{s_j}$ are contained in $M$ because of Lemma~\ref{lemma_main_1}.
Since $j \notin T$ and $s_j \cdot \swh{s_j} \in M$ as seen in~\eqref{eq_si_swh_si_in_M}, we obtain
\begin{equation}\label{eq_trivial_case_2}
\rho_j - \rho_{j-1} \in M \quad \text{ for }t_a < j < t_{a+1}-1.
\end{equation}

For the index $j = t_{a+1}-1$, the module $M_{j,h}$ is spanned by
\[
\underbrace{\{ \swh{u} \mid u \in P_j \}}_{\text{$n-j$ elements}} \sqcup
\underbrace{\{ s_{j} \cdot \swh{s_{j}}, (s_{j-1} s_{j})\cdot
\swh{s_{j}},\dots,(s_{1} \cdots s_{j})\cdot \swh{s_{j}}\}}_{\text{$j$
elements}}.
\]
A similar computation in~\eqref{eq_sj_sigma_sj_and_rho} leads to
\[
s_{j} \cdot \swh{s_{j}}
\equiv \rho_{j} - \rho_{j-1} \mod M,
\]
and, by~\eqref{eq_rhoj_relation}, we have the following.
\[
\begin{split}
(s_k s_{k+1} \cdots s_{j+2} s_{j+1} s_{j}) \cdot \swh{s_{j}}  &
= (s_k s_{k+1} \cdots s_{j+2}s_{j+1})\cdot( s_j \cdot \swh{s_j}) \\
&\equiv (s_k s_{k+1} \cdots s_{j+2} s_{j+1}) \cdot (\rho_j - \rho_{j-1}) \mod M \\
&\equiv (s_k s_{k+1} \cdots s_{j+2}) \cdot (\rho_{j+1} - \rho_{j-1}) \mod M \\
& \equiv \rho_{k}- \rho_{j-1} \mod M
\end{split}
\]
for $j \le k \le n-1$.
Accordingly, we obtain
\begin{equation}\label{eq_trivial_case_3}
\rho_{k} - \rho_{t_{a+1}-2} \in M \quad \text{ for } t_{a+1}-1 \le k \le n-1.
\end{equation}
Combining~\eqref{eq_rho_sum_in_M}, ~\eqref{eq_trivial_case_2}, and~\eqref{eq_trivial_case_3}, we obtain $\rho_j \in
M$ for $i \le j <n$, which proves the desired claim when $i = 1 \in T$ and $i+1 \notin T \cup \{n\}$.

\medskip 
\textsf{Case 2-1} and \textsf{Case 2-2} cover the initial case $a = 1$  (that is, $i = t_1 = 1\in T$)  of the induction argument. Now we suppose that $1 \neq i = t_a \in T$, and moreover,   the claim holds for the smaller indices than $a$. Under this setting, we prove the claim for such $a$, i.e., when $i = t_a \neq 1$.

For such an index $i= t_a \in T$, 
because of Lemma~\ref{lemma_main_1} and
the induction hypothesis, the second and third terms on the right-hand side of~\eqref{eq_si_swh_si} are
already contained in $M$. This provides $\tau_i \in M$ and the modular equality~\eqref{eq_case2-1_modulo}.

If $i+1 \in T \cup \{n\}$, then by the same argument we used in \textsf{Case~2-1}, the claim holds. Otherwise, by applying  the same argument we used in \textsf{Case~2-2}, the claim holds.

This completes the proof of the theorem.	
\end{proof}
For $i \notin T$, we have $\hatswh{i} = \swh{\wi{i}}$.
Accordingly, for this case, we can restate Theorem~\ref{thm_H2}\eqref{stab} as follows.
\begin{corollary}\label{cor_stab_nn}
Let $h \colon [n] \to [n]$ be a Hessenberg function. Suppose that $i \notin T = T_h$ and $i+1 \notin T \cup\{n\}$. Then, $\Stab(\swh{\wi{i}}) = \mathfrak{S}_{(n-i,i)}$.
\end{corollary}

\begin{remark}\label{rmk_question}
By Corollary~\ref{cor_stab_nn}, also by Propositions~\ref{prop:stabilizer_12} and~\ref{prop:stabilizer_34}, we explicitly compute the stabilizer $\Stab(\swh{\wi{i}})$ except one case $i \in T$, $i+1 \in T\cup \{n\}$, $2i = n$, and $T \subsetneq [n-1]$.
Moreover, from Definition~\ref{def_hat_wi}, the class $\hats_{i,h} = \swh{\wi{i}}$ if $i \notin T$. Therefore, by Theorem~\ref{thm_H2}(1), the module $M(\swh{\wi{i}}) = M_{i,h}$ becomes a permutation module. This provides an affirmative answer to Question~\ref{question} when $w \in \mathcal{G}_h^1$.
\end{remark}

\begin{example}
Let $h = (2,3,5,6,6,6)$. Following the notation in Theorem~\ref{thm:Chow}, we
obtain
\[
\alpha^1 = (6), \quad \alpha^2 = (2,4), \quad \alpha^3 = (1,5), \quad \alpha^4
= (6), \quad \alpha^5 = (6).
\]
Since $T = \{1,4,5\}$, we have $\hats_{i,h} = \swh{\wi{i}}$ for $i = 2$ or $3$, and
\[
M_{i,h} = \span_{\C}(\{ \swh{u} \mid u \in P_i\} \cup \{ s_i \cdot \swh{s_i}\}).
\]
Accordingly, $\swh{u} \in M = \sum_{i=1}^5 M_{i,h}$ for any $u \in P_2 \sqcup
P_3$ and $s_i \cdot \swh{s_i} \in M$ for $i =2,3$.

We demonstrate that for $u \in P_1$, the class $\swh{u}$  is contained in $M$ in this example.
The element $\hats_{1,h}$ becomes
\[
\hats_{1,h} = \swh{612345} + \swh{512346}+ \swh{412356} +\swh{312456}
+\swh{213456} + s_1 \cdot \swh{s_1}.
\]
Moreover, using Example~\ref{example_235666_Ai} and Corollary~\ref{cor_si_action_on_sigma_si}, we obtain
\[
\begin{split}
s_1 \cdot \swh{s_1} &= \swh{s_1}
\quad \boxed{
\begin{aligned}[t]
&+ \swh{23\vred1456} + \swh{24\vred1356} + \swh{25\vred1346}
+ \swh{26\vred1345} + \swh{234\vred156}\\
&\quad- (\swh{13\vred2456} + \swh{14\vred2356} + \swh{15\vred2346}
+ \swh{16\vred2345} + \swh{134\vred256} )
\end{aligned}
}\\
&\equiv \swh{s_1} \mod M.
\end{split}
\]
Here, we draw a box for presenting elements in $M_{2,h} \sqcup M_{3,h}
\subset M$.
Hence, we get
\begin{equation}\label{eq_example_main_thm_1}
M \ni \hats_{1,h}  \equiv \swh{612345} + \swh{512346}+ \swh{412356} +\swh{312456}
+ 2 \swh{213456} \mod M.
\end{equation}
On the other hand, we get
\begin{equation}\label{eq_example_main_thm_2}
M \ni s_2 \cdot \swh{s_2} \equiv \swh{3\vred12456} - \swh{2 \vred 13456} \mod M,
\end{equation}
and
\begin{equation}\label{eq_example_main_thm_3}
\begin{split}
M \ni s_3 \cdot \swh{s_3}
&\equiv \swh{4\vred12356} - \swh{3\vred12456} \mod M, \\
M \ni (s_4s_3)\cdot \swh{s_3}
&\equiv \swh{5\vred12346} - \swh{3\vred12456} \mod M, \\
M \ni (s_5s_4s_3)\cdot \swh{s_3}
&\equiv \swh{6\vred12345} - \swh{3\vred12456} \mod M.
\end{split}
\end{equation}
Accordingly, by ~equations~\eqref{eq_example_main_thm_2} and~\eqref{eq_example_main_thm_3}, we
obtain
\[
\begin{tikzcd}[column sep = 0cm, row sep = 0cm]
 & & & \swh{3\vred12456} & - \swh{2 \vred 13456} & \in M, \\
 & & \swh{4\vred12356} & -\swh{3\vred12456} & & \in M,\\
 & \swh{5\vred12346} & & - \swh{3\vred12456} & & \in M, \\
\swh{6\vred12345} & & & - \swh{3\vred12456} & &  \in M.
\end{tikzcd}
\]
Because of~\eqref{eq_example_main_thm_1}, for any $u \in P_1$, we have $\swh{u}
\in M$.
\end{example}

\section{Concluding remarks}

Let $h\colon [n] \rightarrow [n]$ be a Hessenberg function.
Recall that we assign an acyclic orientation~$o_h(w)$ of $G_h$ to each $w \in \mathcal G_h$ in Definition  \ref{def_graph_Ghw}. In fact, the orientation $o_h(w)$ we assign to~$w \in \mathfrak S_n$ is opposite to the usual one. More precisely, to be compatible with the previous work in the case when $\Hess(S,h)$ is a permutohedral variety or when the cohomology has degree two, we will consider the source set instead of the sink set.
\begin{lemma}
For $w \in \mathcal G_{h}^k$, write the source set of $o_h(w)$ as $\{\s_1, \dots,\s_{\ell +1 }\}$, where $\s_1> \dots> \s_{\ell +1 } $. Then $w(\s_1) < \dots <w(\s_{\ell+1})$.
\end{lemma}

\begin{proof}
Let $w \in \mathcal G_h^k$.  Suppose that $\s$ and $\s'$ are sources of $o_h(w)$ with   $w(\s)<w(\s')$.
We define a sequence $j_1,\dots,j_{r}$ of indices by
\[
j_x = w^{-1}(w(\s)+x-1) \quad \text{ for }1 \le x \le w(\s')-w(\s) +1 =: r.
\]
In fact, we obtain $w(j_1) = w(\s), w(j_2) = w(\s)+1,\dots,w(j_r) = w(\s) + (w(\s')-w(\s)+1) -1 = w(\s')$.
By the definition of $\mathcal G_h^k$,
\begin{equation}\label{eq_lemma71}
j_x \to j_{x+1} \text{ is an edge in }
G_{w,h} \quad \text{ if } j_{x+1} > j_x.
\end{equation}
Using this property, we will show that $j_x >\s'$ for any $1 \leq x \leq r-1$ varying $x$ from $r-1$ to~$1$ inductively.

For $x = r-1$, since $\s'=j_r$ is a source and $w(j_{r-1})<w(j_r)$, there is no edge in $G_h$ connecting $j_{r-1}$ and $s'$. In particular, by~\eqref{eq_lemma71}, we have $j_{r-1} > \s'$. If $j_{x+1} >\s'$ and $j_{x}<\s'$ for some $1 \le x < r-1$, then by~\eqref{eq_lemma71}, $j_x$ and $j_{x+1}$ is connected by an edge in $G_h$, and so is $j_x$ and $\s'$  because of  $j_x <\s' <j_{x+1}$, a contradiction. Therefore, $j_x >\s'$ for any $1 \leq x \leq r-1$. Consequently, we have $\s >\s'$.
\end{proof}

\begin{definition}\label{def_widehat_swh}
   For $w \in \mathcal G_{h}^k$, we define an element $\widehat{\sigma}_{w,h}$ in $M(\sigma_{w,h})$ as follows.
\begin{enumerate}
\item  Let $\{\s_1> \dots>\s_{\ell +1 }\}$ be the source set
  of the acyclic orientation $o_h(w)$ associated with $w$.    Define $K_i\colonequals  \{w(\s_{i } ), w(\s_i)+1, \dots, w(\s_{i+1 })-1 \}$ for $1 \leq i \leq \ell $. and $K_{\ell+1}\colonequals  \{w(\s_{\ell+1 }), \dots, n\}$.
Denote by     $\mathfrak S_{{\bf K}_{w,h}}$ the Young subgroup  $\mathfrak S_{K_1} \times \dots \times \mathfrak S_{K_{\ell+1}}$.

\item  Define an element $\widehat{\sigma}_{w,h} $ in $ M(\sigma_{w,h})$ by
\[
\widehat{\sigma}_{w,h} \colonequals  \sum_{u \in \mathfrak S_{{\bf K}_{w,h}} } u \cdot \sigma_{w,h}.
\]
\end{enumerate}
\end{definition}

\begin{conjecture}\label{conj:permutation module decomposition_concluding remark}
 For any $w \in \mathcal G_h$,
 the $\mathfrak S_n$-module $M(\widehat{\sigma}_{w,h})$ generated by $\widehat{\sigma}_{w,h}$ is a permutation module and
\[
H^{2k}(\Hess(S,h)) = \bigoplus_{w \in \mathcal G_h^k} M(\widehat{\sigma}_{w,h}).
\]
\end{conjecture}

We remark that $\widehat{\sigma}_{w,h}$ in Definition~\ref{def_widehat_swh} is a constant multiple of the class we considered when $\Hess(S,h)$ is a permutohedral variety (see ~\cite{CHL}) or when $k=1$ (see Definition \ref{def_hat_wi}). 


\providecommand{\bysame}{\leavevmode\hbox to3em{\hrulefill}\thinspace}
\providecommand{\MR}{\relax\ifhmode\unskip\space\fi MR }
\providecommand{\MRhref}[2]{%
  \href{http://www.ams.org/mathscinet-getitem?mr=#1}{#2}
}
\providecommand{\href}[2]{#2}

\end{document}